\documentclass{amsart}
\usepackage[a4paper, left=2.5cm, right=2.5cm, top = 2.8cm, bottom = 2.8cm ]{geometry}
\usepackage{amsmath, amsthm, amssymb, amsfonts}
\usepackage[all,cmtip]{xy}
\usepackage[mathscr]{euscript}
\usepackage{cite} 
\usepackage{hyperref}
\usepackage{verbatim}
\usepackage{enumitem}
\usepackage{tikz-cd}
\usepackage{rotating}
\usepackage[toc]{appendix}

\newtheorem{teo}{Theorem}

\newtheorem{theorem}{Theorem}[section]
\newtheorem{lemma}[theorem]{Lemma}
\newtheorem{proposition}[theorem]{Proposition}
\newtheorem{conjecture}[theorem]{Conjecture}
\newtheorem{corollary}[theorem]{Corollary}

\theoremstyle{definition}
\newtheorem{definition}[theorem]{Definition}
\newtheorem{example}[theorem]{Example}

\theoremstyle{remark}
\newtheorem{rmk}[theorem]{Remark}

\numberwithin{equation}{section}

\newcommand{\cC}{\mathcal{C}}
\newcommand{\cD}{\mathcal{D}}
\newcommand{\cH}{\mathcal{H}}
\newcommand{\cM}{\mathcal{M}}
\newcommand{\cN}{\mathcal{N}}
\newcommand{\cX}{\mathcal{X}}
\newcommand{\cY}{\mathcal{Y}}
\newcommand{\cZ}{\mathcal{Z}}
\newcommand{\cI}{\mathcal{I}}
\newcommand{\cJ}{\mathcal{J}}

\newcommand{\cRR}{\mathcal{R}}
\newcommand{\cR}{k}
\newcommand{\cS}{\mathcal{S}}
\newcommand{\cU}{\mathcal{U}}
\newcommand{\cW}{\mathcal{W}}

\newcommand{\mcM}{\mathscr{M}}

\newcommand{\Sets}{\mathrm{Sets}}
\newcommand{\sS}{\mathcal{S}}
\newcommand{\sSGd}{sG\mbox{-}\Sets_{d}}
\newcommand{\sMod}{s\mathscr{M}\mathrm{od}}
\newcommand{\Mod}{\mathscr{M}\mathrm{od}}

\newcommand{\Psh}{\mathrm{PSh}}
\newcommand{\Sh}{\mathrm{Sh}}

\newcommand{\coCAlg}{\mathrm{coCAlg}}
\newcommand{\coCAlgM}{\mathrm{coCAlg}(\mcM_\cR)}
\newcommand{\coCAlgC}{\mathrm{coCAlg}_{\cR}(\cC)}

\newcommand{\scoCAlg}{s\mathrm{coCAlg}_{\cR}(Sm_F)}
\newcommand{\coCAlgtr}{\mathrm{coCAlg}^{tr}_{\cR}(Sm_F)}
\newcommand{\scoCAlgtr}{s\mathrm{coCAlg}^{tr}_{\cR}(Sm_F)}

\newcommand{\Hom}{\mathrm{Hom}}
\newcommand{\Spec}{\mathrm{Spec}}

\newcommand{\PST}{\mathrm{PST}(Sm_F,\cR)}

\newcommand{\mapmod}{\mathrm{Map}_{s\Mod_\cRR(Sm_F)}}
\newcommand{\mapGd}{\mathrm{Map}_{\Psh(Sm_F,\sSGd)}}

\newcommand{\colim}{\operatorname{colim}}		 
\renewcommand{\lim}{\operatorname{lim}}

\begin{document}

\title{Rational and $p$-local motivic homotopy theory}

\author{Gabriela Guzman}
\address{Universit\"at Duisburg-Essen\\
Thea Leymann str.9\\ 
45127 Essen\\
Germany}
\email{gabriela.guzman@uni-due.de}



\date{September 12, 2019} 



\begin{abstract}

Let $F$  and $k$ be  perfect fields. The main goal of this paper is to investigate algebraic models for the Morel-Voevodsky unstable motivic homotopy category $\mathrm{Ho}((F)$ after  $\mathbf{H}^{\mathbb{A}^1}k$ localization. More specifically, we extend results of Goerss to the $\mathbb{A}^1$-algebraic topology setting:  we study the homotopy theory of the category $s\coCAlg_k(Sm_F)$ of presheaves of simplicial coalgebras over a field $k$ and their $\tau$ and $\mathbb{A}^1$-localizations. For $k$ algebraically closed, we show that the unit of the adjunction $k^{\delta}[-]\dashv(-)^{gp}$ determines the $\mathbf{H}^{\mathbb{A}^1}k$ homotopy type, where $k^{\delta}[-]$ is the canonical coalgebra functor induced by the diagonal map $\Delta:\cX\rightarrow \cX\times \cX$. We extend this result for the category of presheaves of coalgebras over a non-algebraically closed field $k$ and the category of discrete $G$-motivic spaces, for $G=Gal(\bar{k}/k)$.


On the other hand, we show that the category of coalgebra objects in $\PST$ is locally presentable,  where $\PST$ is the category of  presheaves with Voevodsky transfers  and the monoidal  structure is given by a Day convolution product.
\end{abstract}

\maketitle

{\hypersetup{linkcolor=black}

\setcounter{tocdepth}{1}

\tableofcontents}

\section{Introduction}
\it{Motivation:} \rm One of the problems in classical algebraic topology is to find good algebraic invariants. Even in classical algebraic topology, homotopy groups $\pi_{*}(X)$, homology groups $H_{*}(X,\cR)$, cohomology rings $H^{*}(X,\cR)$ and Steenrod operations $Sq^{i}:H^{n}(X,\mathbb{Z}/2\mathbb{Z})\rightarrow H^{n+i}(X,\mathbb{Z}/2\mathbb{Z})$ are  not sufficient to distinguish homotopy types. It is possible to construct examples where two spaces have isomorphic singular homology groups, but their cohomology rings are not isomorphic, or two spaces with isomorphic cohomology rings but with different Steenrod operations. Finally, it is possible to construct spaces with isomorphic cohomology rings and Steenrod operations but with different Massey products. Massey products are a consequence of the existence of chain level multiplication on $C^{*}(X,\mathbb{Z})$, this suggests that Differential Graded Algebras or dually Differential Graded Coalgebras are finer algebraic invariants.

Formally the problem is, for detecting the homology localization, whether there exists a category $\cD$ of algebraic nature, \it{e.g};  \rm  group objects, ring objects, algebra objects or coalgebra objects, and a functor 
\[F:\mathrm{Ho}(\sS)\rightarrow \cD\] such that $F$ is fully faithful. 

More concretely,  this problem was studied by Quillen \cite{MR0258031} and Sullivan \cite{MR0646078}, \cite{MR0425956}  for rational coefficients, and later by Goerss \cite{MR1363853} for coefficients in an arbitrary field. Specifically, Quillen showed that there exists an equivalence among the categories of simply connected spaces, $1$-reduced differential graded Lie algebras and $2$-reduced Differential Graded Coalgebras.

\[\mathrm{L}_{H_{*}\mathbb{Q}}\mathrm{Ho}(\sS_{\geq 1})\rightarrow \mathrm{Ho}(DGL_{\mathbb{Q},\geq 1})\rightarrow \mathrm{Ho}(DGC_{\mathbb{Q},\geq 2}).\]

On the other hand, Sullivan proved that there is an equivalence between the category of nilpotent spaces of finite $\mathbb{Q}$-type, \it{i.e} \rm nilpotent spaces such that $H^{*}(X,\mathbb{Q})$ is of finite dimension, and the subcategory of $DGA_{\mathbb{Q}}$ such that $A^n$ is a $\mathbb{Q}$-vector space of finite dimension.
\[A_{PL}:\mathrm{L}_{H^{*}\mathbb{Q}}\mathrm{Ho}(\sS^{fin,Nil})\rightarrow \mathrm{Ho}(DGA^{fin}_{\mathbb{Q}}).\]

 Goerss avoided those conditions and used the category of simplicial coalgebras $s\coCAlg_\cR$ to show that the canonical chain complex functor at the simplicial level, and with a coalgebra structure induced by the diagonal map $\Delta:X\rightarrow X\times X$, induces, for $\cR$ an algebraically closed field, a fully faithful functor
 \begin{equation}\label{Goerssthm}
 \mathrm{L}_{H_{*}\cR}\sS\rightarrow s\coCAlg_{\cR}.
 \end{equation} The key ingredient in Goerss' approach is a good understanding of the category of coalgebras over an algebraically closed field, provided in \cite{sweedler1969hopf}. The condition of $\cR$ being an algebraically closed field is a strong condition. Goerss refines this functor to the category of simplicial coalgebras over an arbitrary field $\cR$. For that, he uses an intermediate category of spaces with a group action of the absolute Galois group $G=Gal(\bar{\cR}/\cR)$. Here it becomes fundamental to understand the homotopy fixed points functor, which was studied by Goerss in \cite{MR1320993}.

Let us now turn to the algebraic geometry setting and fix $F$ a perfect field. Motivic homotopy theory was introduced by Morel and Voevodsky in their foundational paper  \cite{MorelVoevodsky}, where topological spaces are replaced by presheaves of spaces on the category of smooth schemes of finite type over $F$. They constructed the unstable motivic homotopy category as the $\mathbb{A}^1$-localization of the Jardine model structure  $\mathrm{L}_{Nis}s\Psh(Sm_F)$.  

For motivic spaces we have two candidates for singular homology, $\mathbb{A}^1$-homology  and Suslin homology.  In \cite{MorelFabien2012A1at}, the $\mathbb{A}^1$-homology sheaves $\mathbf{H}^{\mathbb{A}^1}(\cX)$ were introduced; this homology theory could be a priori naive, but actually computes a wealth of important information. 

In  \cite{MR1376246} Suslin and Voevodsky  introduced a singular homology theory for algebraic varieties known as Suslin homology. The idea of their definition is based on the Dold-Thom theorem. Let  $X$ be a scheme of finite type over a field $F$ and $\Delta^{\bullet}_F$ is the cosimplicial scheme, with $n$-cosimplices $\Delta^n_F=Spec\mbox{ }F[t_0,t_1,\cdots, t_n]/(\sum t_i-1)$. The Suslin homology group $H_{i}^{Sus}(X,\mathbb{Z})$ is defined as  $\pi_{i}$ of the simplicial abelian group 

\[Hom(\Delta^{\bullet}_F,\coprod_{d=0}^{\infty}S^d(X))^{+}.\] Furthermore, they note that after inverting $p$, with $p$ being  the exponential characteristic of $F$, $H_{*}^{Sus}(X)$ coincides with 
\[\pi_{*}(C_{*}^{Sus}(X))=H_{*}(C^{Sus}_{*}(X), d=\sum(-1)^{i}\delta_i),\] where $C_{n}^{Sus}(X)$ is the simplicial abelian group generated by closed integral subschemes $Z\subset \Delta^{n}_F\times X$ such that $Z\rightarrow \Delta^n_F$ is finite and surjective. In other words, Suslin homology is the homology of the global sections of the Suslin complex.

Consider the cannonical map adding transfers, with coefficients in a commutative ring $k$,
\[\gamma_{tr,k}:\mbox{Psh}(\mathrm{Sm}_F)\rightarrow \PST.\] For each $\cX\in Spc_{\bullet}(k)$, we can define the Suslin homology sheaves, with $k$-coefficients, are defined as homology sheaves of the Suslin simplicial complex $C^{Sus}_{\bullet}(\cX)$, which is defined as the simplicial presheaf:

\[U\mapsto \gamma_{tr,k}\cX(\Delta^{\bullet}\times U)\] where $\Delta^{\bullet}$ is the cosimplicial scheme, with $n$-cosimplices $\Delta^n=Spec\mbox{ }F[t_0,t_1,\cdots, t_n]/(\sum t_i-1)$




The canonical map adding transfers induces a morphism between the homology sheaves
\[\mathbf{H}^{\mathbb{A}^1}_{*}(\cX,k)\rightarrow \mathbf{H}^{Sus}_{*}(\cX, k).\]  In general, this map is not an isomorphism. For integral coefficients, by results of Morel \cite[Theorem 6.40]{MorelFabien2012A1at},  $\mathbf{H}^{\mathbb{A}^1}_{*}(\mathbb{G}_m)^{\wedge n}$ is related to Milnor-Witt $K$-theory and by results of Suslin and Voevodsky \cite[Theorem 3.4]{MR1744945} $\mathbf{H}^{Sus}_{*}(\mathbb{G}_m)^{\wedge n}$ is closely related to Milnor $K$-theory. 

On the other hand, by \cite[Corollary 16.2.22]{2009arXiv0912.2110C}  we have an equivalence of symmetric monoidal triangulated categories
\[\mathbf{D}_{\mathbb{A}^1,\acute{e}t}(F,\mathbb{Q})\simeq \mathbf{DM}_{\acute{e}t}(F,\mathbb{Q})\] where  $\mathbf{D}_{\mathbb{A}^1,\acute{e}t}(F,\mathbb{Q})$ is the  \'etale $\mathbb{A}^1$-derived category with $\mathbb{Q}$-coefficients and  $\mathbf{DM}_{\acute{e}t}(F,\mathbb{Q})$ is Voevodsky's triangulated category of motives with $\mathbb{Q}$-coefficients. Then, after taking the version of $\mathbf{H}^{\mathbb{A}^1}_{*}(\cX)$ for the \'etale topology and using $\mathbb{Q}$-coefficients, the $\mathbb{A}^1$-homology is the same as rational Suslin homology.

\bigskip

\it{Main Content:} \rm For $k$ and $F$ perfect fields, we investigate algebraic models for $\mathrm{L}_{\mathbf{H}^{\mathbb{A}^1}\cR}\cH_{\bullet}(F)$  in terms of homotopy categories of coalgebras, where  $\mathbf{H}^{\mathbb{A}^1}\cR$ is the $\mathbb{A}^1$-homology theory of Morel  with coefficients in $k$ and $\mathrm{L}_{\mathbf{H}^{\mathbb{A}^1}\cR}\cH_{\bullet}(F)$ is the localization of the unstable motivic homotopy category with respect to these homology theory. The results are quite parallel to the results of Goerss. The diagonal map $\Delta:\cX\rightarrow \cX\times\cX$ induces a coalgebra structure in $C^{\mathbb{A}^1}_{\bullet}(\cX,\cR)$, which is a presheaf of coalgebras.  By results of  \cite{MR3073905}, we know that   the category of presheaves of simplicial coalgebras   $s\coCAlg(Sm_F)$ is endowed with a left proper, simplicial, cofibrantly generated model category structure, which is left induced from a combinatorial model structure in $s\Mod_{\cR}(Sm_F)$. We extend Goerss' results and we show:
\begin{teo}[\ref{GoerssThm}]\label{GoerssA1thm}
Let $k$ be an algebraically closed field. Then the functor below is fully faithful 
\[\cR^{\delta}[-]:\mathrm{L}_{\mathbf{H}^{\mathbb{A}^1}\cR}\mathrm{Ho}(\mathrm{L}_{mot}s\Psh(Sm_F))\rightarrow \mathrm{Ho}(\mathrm{L}_{mot}\scoCAlg).\]  Furthermore, for every motivic space $\cX$ the derived unit map 
\[X\rightarrow (\cR^{\delta}[\cX]^{fib})^{gp} \] exhibits the target as the $\mathbf{H}^{\mathbb{A}^1}\cR$ localization of $\cX$. 
\end{teo}

If $k$ is non-algebraically closed, we need an auxiliary category: Let $\bar{k}$ be the algebraic closure of a perfect field $k$ and  $G=\mathrm{Gal}(\bar{k}/k)$. We define the category of discrete $G$-motivic spaces as:

\[Spc^{G}_{\bullet}(F):=\mbox{L}_{\mathbb{A}^1}\mbox{L}_{G\times Nis}s\Psh_{inj}(\mathrm{Orb}(G)\times \mathrm{Sm}_F).\] We define the functor \[\bar{k}^{\vee}[-]_G:Spc^{G}_{\bullet}(F)\rightarrow\scoCAlg \] by sending each representable sheaf $G/H$ to the constant simplicial presheaf of coalgebras $(\bar{k}^H)^\vee$ placed on simplicial degree $0$, where $(\bar{k}^H)^\vee$ is the $k$-dual of $\bar{k}^H$.

The functor defined above has a right adjoint $R:\scoCAlg\rightarrow Spc^{G}_{\bullet}(F)$, such that each $RC_n(U)$ indexes all the embeddings of the coalgebra $(\bar{k}^H)^{\vee}$ in $C(U)$, where $H$ runs in all the subgroups $H<G$ of finite index.

\begin{teo}\label{thm2}
Let $k$ be a perfect field. The functor $\cR^{\vee}[-]_G$ induces a fully faithful functor in the homotopy categories:
\[\mathrm{L}\bar\cR^{\vee}[-]_G:\mbox{\rm L}_{\mathbf{H}^{\mathbb{A}^1}\cR}\mathrm{Ho}((Spc^{G}_{\bullet}(F))\rightarrow \mathrm{Ho}((\mathrm{L}_{\mathbb{A}^1}\mathrm{L}_{Nis}s\coCAlg_\cR(\mathrm{Sm}_F)).\]
Furthermore, for every motivic space $\cX$ the derived unit map 
\[\cX\rightarrow R((\bar\cR^{\vee}[\cX]_G)^{fib}) \] exhibits the target as the $\mathbf{H}^{\mathbb{A}^1}\cR$ localization of $\cX$ in discrete $G$-motivic spaces.  

\end{teo}

Furthermore, for $\mathbb{Q}$ coefficients we expect to have the folowing result:
\begin{conjecture}\label{QGoerssthm}

Let $\mathbf{H}^{\mathbb{A}^1}\mathbb{Q}$ be the rational $\mathbb{A}^1$-homology and $\cX$ an $\mathbb{A}^1$-nilpotent space. Let $\cY$ be the $\mathbf{H}^{\mathbb{A}^1}\mathbb{Q}$-localization of $\cX$ as described in Theorem \ref{thm2}. Then  $Y^{hG}$ is the $\mathbf{H}^{\mathbb{A}^1}\mathbb{Q}$-localization of $\cX$ as a motivic space, where $\cY^{hG}$ is the homotopy fixed point space. 
\end{conjecture}

The previous results are not restricted to the  Nisnevich topology. Using the \'etale version of the $\mathbb{A}^1$-homology with rational coefficients $\mathbf{H}^{\mathbb{A}^1}_{*,\acute{e}t}\mathbb{Q} $, we can interpret our results as giving a corresponding localization functor for the rational  Suslin homology localization

In an upcoming project, we plan to study an algebraic model for the Suslin homology localization $\mathbf{H}^{Sus}\cR$.  Our first result is about the local presentability of the category of coalgebras with transfers. More precisely, consider the category of presheaves with transfers $\PST$  (see \cite[Definition 2.1]{MR2242284}), this category is endowed with a monoidal structure, which is given by a Day convolution product. We denote by  $\scoCAlgtr$ the category of coalgebra objects in $\PST$.

\begin{teo}[Corollary  \ref{lpres} and  Theorem\ref{DayFTC}]\label{CoalgDay}
The category of coalgebras  $\coCAlgtr$ is locally presentable, with strong generators given by $\{F\in \coCAlgtr: \#(F)\leq \mbox{max}(\#(C),\aleph_0 )\}$. Furthermore, the underlying functor $U:\coCAlgtr\rightarrow \PST$ has a right adjoint and is comonadic.
\end{teo}

\bigskip

\it{Outline of this work:}   \rm In  Section \ref{STC}, we recall the structure theory of coalgebras over a field following \cite{sweedler1969hopf} and \cite{Nikolaus}. We discuss the category of presheaves of coalgebras for the Day convolution product, here we prove Theorem \ref{CoalgDay}. In Section\ref{MotSp}, we recall some well-known properties of the category of motivic spaces.  In section \ref{leftcoal}, we discuss the homotopy theory for presheaves of simplicial coalgebras, and rely on ideas from \cite{MR3073905} to study the $\mathbb{A}^1$-localization of $s\coCAlg_\cR(Sm_F)$; here we proved Theorem \ref{GoerssA1thm}.  In section\ref{gmot}, we introduce the notion of $G$-discrete motivic spaces and we study the notion of homotopy fixed points. Here we prove Theorem \ref{thm2} and Theorem \ref{QGoerssthm}.

\it{Acknowledgments.} \rm This work is part of the author's PhD thesis at the University Duisburg-Essen under the supervision of Marc Levine. The author is heartily thankful for his constant encouragement and numerous conversation about this work. The author is very grateful to Jens Hornbostel for careful reading of a draft of this paper and his various remarks and corrections. The author also wants to express her gratitude to Lorenzo Mantovani for helpful conversations about this work and mathematics. This work was support by the DFG Schwerpunkt Programme 1786 Homotopy Theory and Algebraic Geometry.

\section{Structure Theory of Coalgebras}\label{STC}

Let $\cRR$ be a commutative ring and $(\mcM_\cR\Delta_CR,\otimes_\cRR, 1_\mcM)$ an $\cRR$-linear symmetric monoidal category. To an $\cRR$-linear symmetric monoidal category we associate the category of cocommutative, coassociative, counital $\cRR$-coalgebras with respect to the monoidal pairing, we denote this category  as  $\mathrm{coCAlg}(\mcM_\cRR)$. More explicitly an $\cRR$-coalgebra $(C,\Delta_C,\varepsilon_C)$ is an object $C\in \mcM_\cRR$ together with $\cRR$-linear morphisms 
\[\Delta_C: C\rightarrow \cC\otimes_\cRR C\] 
\[\varepsilon_C:C\rightarrow \cRR\] such that the following diagrams commutes:

\bigskip

\bigskip

\begin{minipage}[t]{0.3\textwidth}

\[\xymatrix{
C\ar[r]^-{\Delta_C}\ar[rd]_-{\Delta_C}&C\otimes_\cRR C\ar[d]^-{\tau}\\
&C\otimes_\cRR C
}\]
\centering
\it Cocommutativity
\end{minipage}
\begin{minipage}[t]{0.3\textwidth}

\[\xymatrix{
C\ar[r]^-{\Delta_C}\ar[d]_-{\Delta_C} &C\otimes_\cRR C\ar[d]^-{\Delta_C\otimes 1}\\
C\otimes_\cRR C \ar[r]_-{1\otimes \Delta_C}& C\otimes_\cRR C\otimes_{\cRR} C
}\]
\centering
\it Coassociativity
\end{minipage}
\begin{minipage}[t]{0.3\textwidth}      
        
\[\xymatrix{
&C\ar[dl]_-{\simeq}\ar[d]^{\Delta_C}\ar[dr]^-{\simeq}
&\\
C\otimes_{\cRR} \cRR 
& C\otimes_\cRR C\ar[l]^-{1\otimes \varepsilon_C} \ar[r]_-{\varepsilon_C\otimes 1}
&\cRR\otimes_\cRR C}\] 
\centering
\it Counitality       
\end{minipage}

\bigskip

\bigskip

A morphism between $\cRR$-coalgebras $C,D\in\coCAlgM$ is a $\cRR$-linear morphism that it is compatible with the structure maps $\Delta$ and $\varepsilon$,  \it{i.e} \rm the following diagrams commute:

\bigskip

\begin{minipage}[t]{0.5\textwidth}

\[\xymatrix{
C\ar[r]^-{\varphi}\ar[d]_-{\Delta_C}
&D\ar[d]^-{\Delta_D}\\
C\otimes_\cRR C\ar[r]^-{\varphi\otimes \varphi}
&D\otimes_\cRR D}\]
\end{minipage}
\begin{minipage}[t]{0.5\textwidth}

\[\xymatrix{
C\ar[rr]^-{\varphi}\ar[rd]_-{\varepsilon_C}      &&  D\ar[ld]^-{\varepsilon_D}\\
&k&}\]
\end{minipage}

\begin{example}
\begin{enumerate}

\item For  $\mcM=\Mod_\cRR$ we get  $\mbox{coCAlg}_\cRR=\mbox{coCAlg}({\Mod_\cRR})$

\item \label{psh} Let $\cC$ be a small category and $\mcM=\Mod_\cRR(\cC):=\mbox{Fun}(C^{op}, \,\cRR\mbox{-}\Mod)$, the category of presheaves of modules. This category is endowed with the  \it sectionwise tensor product \rm which induces a symmetric  monoidal structures. Let us denoted by $\coCAlgC$ the category of coalgebras.
\item A particular case of the previous example is formed by considering the category $\Delta$ of finite ordered sets and the section-wise tensor product. This so-called category of simplicial coalgebras is denoted by $s\mbox{coCAlg}_\cR$.

\item Let $\cC$ be a small category and $s\mcM=s\Mod_\cRR(\cC):=\mbox{Fun}(\Delta^{op}, \, \cRR\mbox{-}\Mod(\cC))$, the category of simplicial presheaves of modules we write $s\mbox{coCAlg}_\cRR(\cC)$ for the category of presheaves of simplicial coalgebras.

\end{enumerate}
\end{example}


\subsection{The category of coalgebras over a field $\coCAlg_k$}\label{coalgclassif}
We now discuss the structure of the category of coalgebras over a field $k$. These results are well known, a good reference is \cite{sweedler1969hopf} and Theorem \ref{etalesplit} is proved in \cite{MR1363853} for the case of an algebraically closed field. In \it{loc.cit.} \rm Goerss claims that the theorem is valid for $k$ a perfect field, a proof is given in \cite{Nikolaus}.

First, let us recall the definition of a locally finitely presentable category.
\begin{definition}
Let $\cC$ be a small category. An object $C\in \cC$ is compact or finitely presentable if the representable functor $\cC(C,-)$ preserves filtered colimits.  
\end{definition}

\begin{definition}
A small category $\cC$ is locally finitely presentable if $\cC$ has all small colimits, the subcategory of compact objects $\cC_{fp}$ is essentially small and every object $C\in \cC$ is a filtered colimit of the canonical diagram of the finitely presentable objects mapping into it.
\end{definition}

\begin{proposition}[Fundamental Theorem of coalgebras]\label{FTCsincolim}
Let  $C$ be a coalgebra over a field $k$  and $x\in C$. Then there exists a finite dimensional subcoalgebra $D\subset C$ with $x\in D$.
\end{proposition} 
\begin{proof}
Given a basis for $C$ we can express $\Delta_C(x)=\sum x_i\otimes c_i$ with $c_i$ elements in the basis, and $(\Delta_C\otimes id)(\Delta_C(x))=\sum\Delta_C(x_i)\otimes c_i=\sum a_j\otimes b_{i,j}\otimes c_i$ with  $a_j$ and $c_j$  linear independent. Then define $D=\langle b_{i,j}\rangle$ as the subspace generated by $b_{i,j}$.  By the properties of the  coassociativity, counitality and  cocommutativity of the coalgebra structure, it is not difficult to prove that $D$ is a subcoalgebra. 
\end{proof}

As consequence of Proposition \ref{FTCsincolim} we have 
\begin{proposition}\label{FTC}
Every coalgebra over a field $\cR$ is the filtered colimit of its finite dimensional sub coalgebras. 
\end{proposition}  
\begin{proof}
From \ref{FTCsincolim} we have that the underlying vector space $C$ is the colimit of a diagram $C_\alpha$, where each $C_\alpha$ is a finite dimensional coalgebra. It remains to show that the colimit of a diagram of coalgebras is again a coalgebra. This last statement is straightforward.  \end{proof}

\begin{theorem}\label{cofree}
The category $\mathrm{coCAlg}_k$ for $k$ a field  is finitely presentable and the forgetful functor $U:\mathrm{coCAlg}_k\rightarrow \Mod_k$ has the right adjoint $CF$. For a $k$-vector space $V$, $CF(V)$ is the so-called cofree coalgebra on $V$. 
\end{theorem}
\begin{proof}
Since  $-\otimes-$ preserves colimits, the forgetful functor creates colimits. Then, considering \ref{FTC}, we can conclude that  $\mathrm{coCAlg}_k$ is finite presentable. To prove that the forgetful functor has a right adjoint, we need to apply the dual version of the Special Adjoint Functor Theorem \cite{MacLaneSaunders1998}. It suffices to verify that the category   $\mathrm{coCAlg}_k$ is well-copowered, which means that the collection of quotients is a set. This last condition follows from the fact that $\Mod_K$ is well copowered. Then the forgetful functor has a right adjoint: the cofree algebra functor.
\end{proof}

We have already mentioned the existence of colimits. Since the category of coalgebras is finitely presentable, small  limits also exist, because every finitely presentable subcategory is a reflective subcategory of a category of presheaves \cite[Corollary 1.28]{jiris1994}.

\begin{definition}
A coalgebra $C$ over a field $k$ is called simple if it has no non-trivial subcoalgebras. 
\end{definition}
 Let $K/k$ be a finite field extension. Then $K^{\vee}=\Hom_k(K,k)$ is a finite dimensional coalgebra over $k$, and it is simple because subcoalgebras of $K^{\vee}$ corresponds to quotients of $K$, and only the trivial quotient $K$ exists.
 
 \begin{proposition}
Let $D$ a simple coalgebra over $k$. There exists a finite field extension $K$ over $k$ such that $D\cong K^{\vee}$.
 \end{proposition}    
\begin{proof}
By Proposition \ref{FTCsincolim} every  non-finite dimensional coalgebra over $k$ contains a non-trivial subcoalgebra. Then every simple subcoalgebra $D$ is finite dimensional, and $D^{\vee}$ is a finite-dimensional commutative algebra which has no non-trivial quotients. Then $D^{\vee}$ has only the trivial ideals, thus it is isomorphic to a field. Then it is a finite field extension of $k$. Furthermore, again since $D$ is finite dimensional, $D\cong D^{\vee\vee}$.
 \end{proof}

\begin{rmk}
In particular if $k$ is algebraically closed there are no simple coalgebras over $k$ besides $k$ itself.
\end{rmk}

\begin{definition}
Let $C$ be a coalgebra over $k$,  the \it \'etale  part $\acute{E}t(C)$  \rm of $C$ is the direct sum $\oplus_{C_{\alpha}\subset C }C_{\alpha}$ where $C_\alpha$ runs through all the simple subcoalgebras of $C$.

A coalgebra $C$ over a field $k$ is called \it{irreducible}  \rm   if it contains a unique simple subcoalgebra, \it{i.e.}   \rm if the \'etale part consists only of one single summand.  A coalgebra is called an \it{irreducible component }\rm if it is a maximal irreducible subcoalgebra of $C$.    
 
\end{definition}

\begin{lemma}\label{sumcoalg}
Let $C=\sum_{i\in I}C_\alpha$ be a (not necessarily direct) sum of subcoalgebras $C_\alpha\subset C$. Any simple subcoalgebra of $C$ lies in one of the summands $C_\alpha$.
\end{lemma}
\begin{proof}
Let $D\subset C$ be a simple coalgebra. Since $D$ is finite dimensional, it lies in the sum of a finite number of summands. By induction on $n$, it suffices to prove that if $D\subset C_{\alpha_1}+ C_{\alpha_2}$ then $D\subset C_{\alpha_1}$ or $D\subset C_{\alpha_2}$. Suppose that $D$ is not contained in $C_{\alpha_1}$. Since $D$ is simple,  then $D\cap C_{\alpha_1}=0$.  We can choose a linear map $f:C\rightarrow k$ such that $f|_D=\varepsilon_D$ and $f|_{C_{\alpha_1}}=0$. Every $d\in D$ satisfies   
\begin{equation}
(f\otimes 1)(\Delta_D(d))= (\varepsilon_D\otimes 1)(\Delta_D(d))=d
\end{equation} but $\Delta_D (D)\subset C_{\alpha_1}\otimes C_{\alpha_1}+C_{\alpha_2}\otimes C_{\alpha_2}$. Since $f|_{C_{\alpha_1}}=0$, we conclude that   for every $d\in D$,  $d\in C_{\alpha_2}$.
\end{proof}

\begin{lemma}\label{Naturalet}
The canonical morphism $\acute{E}t(C)\rightarrow C$ is an injective morphism of coalgebras and in fact it defines an endofunctor:
\[\begin{split}
\acute{E}t:\coCAlg_k&\rightarrow \coCAlg_k \\
                              C&\mapsto \acute{E}t(C)
\end{split}\] such that the canonical morphism is natural.
\end{lemma}

\begin{proof}
We claim that the sum of simple coalgebras $\sum C_\alpha\subset C$ is a direct sum. To prove that we have to show that $C_{\alpha_0}\cap \sum_{\alpha\neq \alpha_0} C_{\alpha}=0$ for every $\alpha\neq 0$.  Suppose that $C_{\alpha_0}\cap \sum_{\alpha\neq \alpha_0} C_{\alpha}\neq0$, since $C_{\alpha_0}$ is a simple coalgebra it follows by Lemma \ref{sumcoalg} that $C_{\alpha_0}=C_{\alpha_1}$ for some $\alpha_1$, which is a contradiction. 

It remains to prove the functoriality and the naturality. We claim that given $f:C\rightarrow D$ a morphism of coalgebras  the image $f(C_{\alpha})\subset D$  is simple for every simple subcoalgebra $C_\alpha\subset C$. Since the image is a quotient of $C_\alpha$,  it suffices to showing that the quotient of a simple coalgebra is simple. This is equivalent to showing that the subalgebra of a finite field extension is finite field extension, which is true. 

\end{proof}

\begin{rmk}\label{diagonalgplike}
Let $\cRR$ be a ring. For a set $X$, we let $\cRR[X]$ denote the free $\cRR$-module on $X$.  The coalgebra $R^{\delta}[X]$ is the $\cRR$-module $\cRR[X]$ with coproduct induced by the diagonal map $\Delta:X\to X\times X$ and the isomorphism $\cRR[X\times X]\cong \cRR[X]\otimes_{\cRR}\cRR[X]$. Note that the coalgebra functor
\[\cRR^{\delta}[-]:Sets\rightarrow \coCAlg_{\cRR}\] admits a right adjoint given by 
\[\xymatrix{(-)^{gp}:\coCAlg_{\cRR}\rightarrow Sets& C\mapsto \mbox{Hom}_{\coCAlg_\cRR}(\cRR,C).}\]  The right adjoint $(-)^{gp}$ can be given more explicitly, a morphism of coalgebras $\cRR\rightarrow C$ sends $1$ to an element $c\in C$ such that $\Delta_C(c)=c\otimes c$ and $\varepsilon_C(c)=1$, such elements are known as \it{group like}  \rm (and throughout the text) elements in $C$. Every group-like element determines a unique morphism  $\cRR\rightarrow C$ of coalgebras. Thus the right adjoint is given by sending the coalgebra $C$ to the subset $C^{gp}$.  
\end{rmk}

\begin{proposition}\label{Etgplike}
Let $k$ be an algebraically closed field and $C$ a coalgebra over $k$. Then the counit of the adjunction $k^{\delta}[C^{gp}]\rightarrow C$ factors through the inclusion $\acute{E}t(C)\subset C$ and induces a natural isomorphism
\[k^{\delta}[C^{gp}]\cong \acute{E}t(C).\] 
\end{proposition}

\begin{proof}
Let $X$ be a set and consider $k^{\delta}[X]=\sum_{X}k$ with  the coalgebra structure induced by the diagonal map. We have that $\acute{E}t(k^{\delta}[X])=k^{\delta}[X]$.  Then by Lemma \ref{Naturalet} the counit map $k^{\delta}[C^{gp}]\rightarrow C$ factors through the \'etale part of $\acute{E}t(C)$. Now since $k$ is algebraically closed a simple subcoalgebra of $C$ is given by $k$ and then the \'etale part is given by the direct sum over all morphisms of coalgebras $k\rightarrow C$, which is the description of the group like elements.    
\end{proof}

\begin{lemma}\label{dirsumic}
Every coalgebra over a field $k$ is the direct sum of its irreducible components. In other words for every simple coalgebra $C_{\alpha}$ there is a unique irreducible component $\overline{C}_\alpha\subset C$  such that $C_\alpha\subset  \overline{C}_\alpha$ and the canonical morphism 
$$\bigoplus_{\alpha} \overline{C}_\alpha\rightarrow C$$ is an isomorphism of coalgebras.   
\end{lemma}
\begin{proof}
The sum of all irreducible coalgebras which contains $C_\alpha$ is also irreducible: if another simple subcoalgebra $C_\beta$ is a subcoalgebra of this sum, then it is contained in one of the summands but this is a contradiction because each summand is an irreducible containing  $C_\alpha$.  By construction the sum contains $C_\alpha$ and is maximal then is an irreducible component. 

We show now that he sum of the irreducible components is a direct sum. Suppose  that  there is a non-trivial intersection $\overline{C}_{\alpha_0}\cap \sum_{\alpha\neq \alpha_0}\overline{C}_{\alpha}$ for some $\alpha_0$. Then this intersection contains a simple subcoalgebra which has to be $\overline{C}_{\alpha_0}$, since it is a subcoalgebra of  $\overline{C}_{\alpha_0}$, but again  Lemma \ref{sumcoalg} shows that  $C_{\alpha_0}$ is a subcoalgebra of  $C_{\alpha}$ for some $\alpha\neq \alpha_0$ which is a contradiction.  

We have that $\sum_{\alpha}C_\alpha\subset C$, then it is enough to show that each element $c\in C$ lies in a sum of irreducible coalgebras. Take $\overline{\{c\}}$ the subcoalgebra generated by $c$, which is finite dimensional by the fundamental theorem of coalgebras.  The $A=C^{\vee}$ is an artinian algebra and we have that $A\cong A_1\oplus\cdots\oplus A_n$  with each $A_i$ a local artinian subalgebra, but since $A_i$ is local  $A_i^{\vee}$ is irreducible.  
\end{proof}

\begin{lemma}\label{funic}
Let $f: C\rightarrow D$ be a morphism of coalgebras over a field $k$. Then $f$ restricts to a morphism of irreducible components $\overline{C_\alpha}\rightarrow \overline{f(C_\alpha)}$ for every simple subcoalgebra $C_{\alpha}$ of $C$.
\end{lemma}
\begin{proof}
We first prove the result in case $C$ is irreducible and $f$ is surjective. Let us show  that $D$ is irreducible. By Proposition \ref{FTC}            $C=\colim_{\alpha}C_{\alpha}$  where $C_\alpha$ runs over all the finite dimensional subcoalgebras and thus $D=\colim_{\alpha}f(C_{\alpha})$.

\it{Claim 1:}  \rm  A coalgebra $D$ is irreducible if and only if every element lies in some irreducible subcoalgebra.

\it{Claim 2:} \rm  Let  $C$ be an irreducible finite dimensional coalgebra and $f:C\rightarrow D$ a surjective morphism of coalgebras. Let $C_0$ be the unique simple subcoalgebra of $C$. Then $f(C_0)$ is the unique simple subcoalgebra of $D$.

Granting Claim 1 we can assume that $C$ is finite dimensional and $f$ surjective, then the proposition  for $C$ irreducible and $f$ surjective follows from Claim 2.

\it{Proof of claim 1:}
\rm One direction is obvious. For the converse suppose that $D$ is not irreducible. Then there exists $E$ and $E'$ two distinct irreducible components. We know that $E+ E'$ is a direct sum. Choose $e\in E$ and $e'\in E'$ and consider $e+e'\in E\oplus E'$. Let $F$ be  the subcoalgebra generated by $e+e'$. By hypothesis  $F$ must be contained in an irreducible subcoalgebra, so it is irreducible as well. Let $E_0$ be the simple subcoalgebra of $E$, $E'_0$ the simple subcoalgebra of $E'$, $F_0$ the simple subcoalgebra of $F$. Since, $F$ is contained in $E\oplus E'$, $F_0$ is either $E_0$ or $E'_0$, say $F_0=E_0$. Then $E+F$ is also irreducible and by maximality, $E+F=E$ so $F\subset E$. Thus $e+e'$ is in $E$ and thus $e'$ is in $E$. But then the subcoalgebra generated by $e'$ is contained in $E\cap E'$ and must contain both $E_0$ and $E'_0$, so $E'_0\subset E$, contrary to the assumption that $D$ is irreducible.

\it{Proof of claim2:}
\rm

The exact sequence of coalgebras  $C\rightarrow D\rightarrow 0$ induces an exact sequence of finite dimensional algebras $0\rightarrow D^{\vee}\rightarrow C^{\vee}$. Since $C$ is irreducible $C^{\vee}$ is a finite dimensional local algebra. Let $\mathfrak{m}$ the maximal ideal. By the Nakayama Lemma there exists $n\in \mathbb{N}$ such that $\mathfrak{m}^n=0$. Denote  $\mathfrak{n}= D^{\vee}\cap \mathfrak{m}$, it is an ideal such that $\mathfrak{n}^n=0$, thus lies in the Jacobson radical of $D^{\vee}$. And all maximal ideals of $D^{\vee}$ contain $\mathfrak{n}$. Then all simple subcoalgebras of $D$ are contained in $\mathfrak{n}^{\perp}$. But an explicit computation shows that $\mathfrak{n}^{\perp}= D^{\vee}\cap C_0^{\perp}=f(C_0)$ where $C_0$ is the unique simple subcoalgebra of $C$. Since $f(C_0)$ is a simple subcoalgebra  of $D$ it follows that is the only one. Thus $D$ is irreducible.

We now discuss the general case. By the first part we have that $f(\overline{C_\alpha})$ is irreducible, then it lies in an irreducible component which contains $f(C_\alpha)$. But the former is a simple subcoalgebra. Thus $f(\overline{C_\alpha})\subset \overline{f(C_\alpha)}$.

\end{proof}

\begin{corollary}
Let $k$  be a field. The decomposition of Lemma \ref{dirsumic} is functorial.
\end{corollary}
\begin{proof}
Let $f:C\rightarrow D$ be a morphism of coalgebras. Combining Lemma \ref{dirsumic} and Lemma \ref{funic} we get:
\[\xymatrix{\bigoplus_{\alpha} \overline{C_{\alpha}}\ar[r]^-{\cong} \ar[d]_-{f_{*}}&C\ar[d]^{f} \\
                  \bigoplus_{\alpha}\overline{D_{\alpha}}\ar^-{\cong}[r] & D
                  }\] where $f_{*}$ on the summand $\overline{C_{\alpha}}$ agrees with $f|_{\overline{C_{\alpha}}}$.
\end{proof}
\begin{lemma}\label{retract}

Let $C$ be an irreducible coalgebra over a perfect field. Then there is a unique retract of the inclusion $\acute{E}t(C)\subset C$ that is a map of coalgebras. Moreover, the retraction is natural in $C$.
\end{lemma}
\begin{proof}
Since $C$ is an irreducible coalgebra, $\acute{E}t(C)=K^\vee$ for $K$ a finite field extension of $k$. By the Fundamental Theorem of Coalgebras $C=\colim_i C_i$ where  $C_i\subset C$ are finite dimensional subcoalgebras.  Every $C_i$ is an irreducible coalgebra and contains $\acute{E}t(C)$. Thus, it is enough to show the Lemma for the finite dimensional case. We reduce to show the dual statement: Let $A$ be a  finite dimensional local algebra  over a perfect field $k$, with $\mathfrak{m}$ the unique maximal ideal. Then there exists a unique  subfield $K\subset A$ such that $A= K\oplus \mathfrak{m}$ 

Since $k$ is perfect, the field extension $A/\mathfrak{m}$ is a separable extension over $k$.  By the Primitive Element Theorem, $A/\mathfrak{m}$ corresponds to the form $k(\alpha)$, with  $\alpha\in A/\mathfrak{m}$. Let $p(x)\in k[x]$ be the minimal polynomial of $\alpha$, since it is a separable polynomial over $k$, $p'(\alpha)\neq 0$. Since $A$ is of finite dimension,we have that $\mathfrak{m}^n=\mathfrak{m}^{n+1}$, for some $n\in\mathbb{N}$.  Nakayama lemma implies that $\mathfrak{m}^n=0$. Thus $A$ is complete with respect to the $\mathfrak{m}$-adic topology. By Hensel's Lemma there exists a unique element $x\in\alpha\subset A$ such that $p(x)=0$. Set $K=k(x)\subset A$, this field has the required property since the composition $K\rightarrow A\rightarrow A/\mathfrak{m}$ is an isomorphism by construction. 

It remains to show that the retract is natural.  If $C\rightarrow D$ is a morphism of irreducible coalgebras over $k$, then for the commutative diagram 
\[\xymatrix{\acute{E}t(C)\ar[r]\ar[d]& \acute{E}t(D)\ar[d]\\
                C\ar[r]& D}\]  we need to show that the diagram of retracts $C\rightarrow \acute{E}t(C)$ and $D\rightarrow \acute{E}t(D)$ commutes. 
Recall that any morphism of coalgebras $C\rightarrow D$ can be factored into an epimorphism followed by a monomorphism  $C\twoheadrightarrow F\hookrightarrow D$. Since $C$ and $D$ are irreducible,  Lemma \ref{funic} implies that $F$ is irreducible as well. Thus $\acute{E}t(F)\simeq \acute{E}t(D)$ and we get the following  diagram

\[\xymatrix{\acute{E}t(C)\ar@{->>}[r]\ar[d]&\acute{E}t(F)\ar[r]^-{\cong}\ar[d] & \acute{E}t(D)\ar[d]\\
                C\ar@/^/[u]\ar@{->>}[r]& F\ar@/^/[u]\ar@{^{(}->}[r]& D\ar@/^/[u]}\] 
By the uniqueness of the retraction $F\rightarrow \acute{E}t(F)$, we have that the right-hand side square, formed with the retractions,  commutes. 

It remains to show the the left-hand side diagram commutes. It suffices to reduce again to a finite dimensional case. Thus we have an inclusion of finite dimensional local algebras $A\hookrightarrow B$ and we have to show that the following diagram commutes

\[\xymatrix{A/\mathfrak{m}\ar[r]\ar[d]& B/\mathfrak{n}\ar[d]\\
                A\ar[r]& B}\] Since $\mathfrak{m}=\mathfrak{n}\cap A$ from a similar argument for the uniqueness of the retraction,we get that  
                \[A=B/\mathfrak{m}\cap A\oplus \mathfrak{m}\cap A=k\oplus A\] which implies the claim.
 \end{proof}

\begin{rmk}
Note that for any $D$ simple coalgebra over an algebraically closed field $k$, the counit map $\varepsilon: D\rightarrow k$ is an isomorphism.       

\[\xymatrix{k\cong \acute{E}t(D) \ar@{^{(}->}[r]\ar[d]_-{\cong}&D\ar[d]^-{\varepsilon_D}\\
                k \ar@{=}[r]&k  }\] thus the natural splitting is given by the counit map.    
\end{rmk}                

\begin{theorem}\label{etalesplit}
Let $k$ be a perfect field. Then for every coalgebra $C$ the inclusion  $\acute{E}t(C)\rightarrow C$ has a unique and natural splitting which is a map of coalgebras 
\end{theorem}
\begin{proof}
If $C$ is an irreducible coalgebra the result follows by Lemma \ref{retract}. For the general case by Lemma \ref{dirsumic} we have that $D\cong \oplus_{\alpha} \overline{D_\alpha}$ and $\acute{E}t(D)\cong \oplus_{\alpha} D_\alpha$, where $\{ D_\alpha\}$  is the collection of simple subcoalgebras of $D$. We define the splitting as the direct sum of the splittings for the irreducible components. By Lemma \ref{dirsumic} every simple subcoalgebra is contained in a unique irreducible component. Then there can not be a splitting mixing the components and the naturality follows from Lemma \ref{funic}.      
                        
\end{proof}

\begin{corollary}\label{counitretract}
Let $k$ an algebraically closed field. Then the counit of the adjunction \[k^{\delta}[-]:\Sets\rightleftarrows \coCAlg_{k} :(-)^{gp}\] given by $k^{\delta}[C^{gp}]\rightarrow C$ has a natural retraction. 
\end{corollary}
\begin{proof}
This follows from Proposition \ref{Etgplike} and Theorem \ref{etalesplit}.
\end{proof}

\begin{rmk}
In proposition \ref{Etgplike} the condition that the field is algebraically closed is necessary and it is needed for the proof of Theorem \ref{GoerssThm}. We will introduce a treatment for $k$ non-algebraically closed in chapter 2, we will require an auxiliary category of discrete $G$-sets.
\end{rmk}

\subsubsection{Discrete $G$-$\Sets$}

\begin{definition}
Let $G$ be a profinite group. Then a $G$-set is \it discrete \rm if the action is continuous when $X$ is given with the discrete topology. 
\end{definition}
This is equivalent to ask that the isotropy groups $H_x\subset G$ of $x$ be open, or equivalently that 
\[X=\bigcup_{H\subset G} X^H\] where $H$ runs over all the open subgroups of $G$ and  $X^H$ is the set of fixed points for $H$. In particular the orbit of any $x\in X$ is finite.

We can actually describe the category of discrete $G$-sets in a more sophisticated language. Let us denote by $G$-$\Sets_{fd}$ the full subcategory of finite discrete sets in $G$-$\Sets_{d}$, the category  $G$-$\Sets_{fd}$ has a pretopology defined by the covering families $U_i\rightarrow X$ such that $\coprod_{i} U_i\rightarrow X$ are surjections.  The associated Grothendieck topos is called the \it{classifying topos} \rm for the profinite group $G$, and is denoted as $BG$, \it{i.e} \rm an object in $BG$ is a sheaf of sets on the site $G$-$\Sets_{fd}$.

For each presheaf on $G$-$\Sets_{fd}$, we can define a $G$-set $LF$ as follows:
\[LF:=\colim_{i\in I}F(G_i)\] Right multiplication  by elements of $G_i$ induces a left $G_i$-action on $F(G_i)$ and so there is an induced left $G$-action on $LF$.

\begin{definition}
Let $F\acute{E}t/\cR$ be the full subcategory of $Sm_\cR$ consisting of all the schemes of finite type over $\cR$ which are smooth of dimension zero. Every object $S\in F\acute{E}t/\cR$ is a finite disjoint union of $\Spec(k')$ for $k'$ a finite separable field extension of  $\cR$. 
\end{definition}

Let $U\in F\acute{E}t/\cR$, the functor defined by 
\[U\mapsto \hom_k(Spec(k_{sep}),U) \] defines an isomorphism of sites
\[\mathcal{F}:F\acute{E}t_k\rightarrow G\mbox{-}\Sets_{fd}, \] existance of this isomorphism is just Galois theory and observe for $F$ a sheaf over $G\mbox{-}\Sets_{fd}$ the associated $G$-set $LF$ correspondes to $F_{Spec(k_{sep})}$, the stalk at the geometric point. 

\begin{proposition}
Let $k$ be a separable field and  $G=Gal(k_{sep}/k)$  the absolute Galois group. Then the following categories are equivalent:
\begin{enumerate}
\item The category of discrete $G$-sets.
\item The category of sheaves of sets on $G$-$\mbox{Sets}_{fd}$
\item The category of sheaves of sets  on $F\acute{E}t_k$
\end{enumerate} 
\end{proposition}
\begin{proof}
The equivalence between \it{(1)} \rm and \it(2)  \rm is given for example in \cite[Proposition 6.20]{jardine2010generalized}. The equivalence between  \it{(1)} \rm and \it(3)  \rm follows from the Galois correspondence.

\end{proof}

Let us make some remarks about the topos $G$-$\Sets_{d}$ of discrete $G$-sets
\begin{itemize}
\item The topos of discrete $G$-sets has enought points \it{i.e}  \rm  there is a functor
\[u^{*}:G\mbox{-}\Sets_{d}\rightarrow Sets\] which is defined by forgetting the group structure. Colimits and finite limits  in the category of discrete $G$-sets are formed in the category of sets, then the functor $u^{*}$ is faithful and exact. It is enough to check isomorphism between discrete $G$-sets $F\rightarrow G$  at only one stalk, the underlying set.
\end{itemize}

In particular if $G=Gal(k_{sep}/k)$ for $k$ a field, then we have the following well known identification between the finite \'etale site $F\acute{E}t/\cR$ defined below and the site  $G$-$\Sets_{fd}$ associated to the profinite group $G$.

\subsubsection{The category of coalgebras over a non-algebraically closed field $k$}

\begin{proposition}\label{final}
There exists a left adjoint functor $\bar{k}^{\vee}[-]_G:G-Sets_d\rightarrow \coCAlg_{\cRR}$  with right adjoint given by 
\begin{equation*}
\begin{aligned}
RC:Orb(G)^{op}&\rightarrow Sets  \\
G/H &\mapsto\mbox{Hom}_{\coCAlg_k}((\bar{k}^H)^\vee,C).
\end{aligned}   
\end{equation*} Furthermore, the functor  $\bar{k}^{\vee}[-]_G$ is fully faithful, and the counit of the adjunction is given by the embedding
\[\acute{E}t(C)\rightarrow C\] 

\end{proposition}
\begin{proof}
The functor is defined as the left Kan extension of the functor 
\begin{equation*}
\begin{aligned}
Orb(G)&\rightarrow Sets  \\
G/H &\mapsto(\bar{k}^H)^\vee
\end{aligned}
\end{equation*} along the Yonneda embedding \it{i.e.}\rm,
\[\bar{k}^\vee[X]_G:=\colim_{G/H\rightarrow X}(\bar{k}^H)^\vee.\]  Let us prove that $R$ is its right adjoint. It is enough to construct the unit and counit maps. Note that for every  representable sheaf $G/H$ we have the following isomorphism \[\Hom_{coCAlg_k}((\bar{k}^{H'})^\vee, (\bar{k}^H)^\vee)\simeq Hom(G/H',G/H).\] Then we  have an isomorphism $X\rightarrow R(\bar{k}^\vee[X]_G)$ for each $X\in G\mbox{-}Sets_d$. On the other hand, note that for each $C\in \coCAlg$  $RC(G/H)$ index all the embeddings of $(\bar{k}^H)^\vee$ in $C$, then:
\[\bar{k}^\vee[RC]_G=\colim_{G/H\rightarrow RC}(\bar{k}^H)^\vee=\acute{E}t(C).\] Furthermore, we have a natural embedding  $\acute{E}t(C)\rightarrow C$. A computation shows that the identity transformation $X\rightarrow R(\bar{k}^\vee[X]_G)$ and the inclusion $\acute{E}t(C)\rightarrow C$ defines the unit and counit maps. 
\end{proof}


\subsection{Category of presheaves of coalgebras $\coCAlg_{\cR}(\cC)$}
For this section let us denote $\cC$ an small category and $k$ a field. We extend the theorem \ref{cofree} to the category of presheaves with the sectionwise tensor product as monoidal structure. Here the argument is taken from \cite{MR3073905}. For completeness we reproduce the proof here. \begin{proposition}\label{pshcofreesw}
The category of presheaves of coalgebras $\coCAlg_k(\cC)$ is locally finitely presentable and the forgetful functor $\coCAlg_k(\cC)\rightarrow \Mod_k(\cC)$ has a right adjoint, $CF: \Mod_k(\cC)\rightarrow \coCAlg_k(\cC)$. 
\end{proposition}
\begin{proof}

First assume that $\cC$ is a discrete category, i.e. for each object $U$ the hom sets $\mbox{Hom($U,U$)}=\{id_U\}$  and empty otherwise. By \ref{cofree} the category of presheaves of coalgebras over $\cC$,  is finitely presentable with strong small set of generators $\cD$ given by the collection of presheaves of coalgebras  $D$  such that the sections of $D$ are zero except for just one object $U\in Ob(\cC)$, and $D(U)$ is a finite dimensional coalgebra.  
For the general case let $i:\cC_0\hookrightarrow \cC$ be the canonical inclusion of the discrete subcategory. This inclusion induces an adjunction in the categories of presheaves:
\[i_{!}:\coCAlg_k(\cC_0)\rightleftarrows \coCAlg_k(\cC): i^{*}.\]
The functor $i_{!}$ is given explicitly as:
 \[i_{!}(\cD)(V)=\bigoplus_{V\rightarrow U}\cD(U).\] We want to show that the category $\mbox{ coCAlg}_k$ is finitely presentable, by \cite[Theorem 1.1]{jiris1994} it is enough to show that $\mbox{ coCAlg}_k$ is cocomplete and has a strong set of generators. The counit of the adjunction $i_{!}i^{*}C\rightarrow C$,  with $C\in\mbox{\rm coCAlg}_k$, is a  surjective map  and by the discrete case $i^{*}C=\mbox{colim}_\alpha D_\alpha$ with $D_\alpha\in\cD$. Because $i_{!}$ is a left adjoint, it commutes with colimits and we get a surjective map $\mbox{colim}_\alpha i_{!}(D_\alpha)\rightarrow D$, from this surjection it follows that given two different maps $f,g:D\rightrightarrows E$ there exists a map $i_{!}(D_\beta)\rightarrow D$ such that the composition with $f$ and $g$ is different and for every proper subobject $K$ of $D$, there exists  $i_{!}(D_\gamma)\rightarrow D$ a map such that it does not factorize trough $K$, i.e. $\mbox{colim}_\alpha i_{!}(D_\alpha)$ is an strong generator.  Furthermore $\mbox{ coCAlg}_k$ is cocomplete, colimits are created by the forgetful functor because the section-wise tensor product preserves colimits in both variables, the category is locally finitely presentable. By the dual of the special adjoint functor theorem we have the existence of the cofree algebra functor $CF$. 

\end{proof}

\begin{rmk}
In \it{loc.cit} \rm  the previous statement is proved for $\cRR$ a presheaf of rings on $\cC$, using the local presentability of the category of coalgebras over a ring $\coCAlg_\cRR$ proved in \cite[Theorem 3.1]{BARR1974600}. We do not  need that generality in this work. Moreover, only in the case of algebraically closed fields we have a good description of the category of coalgebras. 
\end{rmk}
\begin{definition} 
Let $C$ be a coalgebra object in $\Mod_{\cR}(\cC)$ under the section-wise monoidal structure. The \'etale subpresheaf is defined as $U\mapsto \acute{E}t(C(U))$.
\end{definition}

The \'etale presheaf is well-defined since for each $U\in \cC$, $C(U)$ is a coalgebra over $k$. Furthermore, the functor $\cR^{\delta}[-]$ extends to a functor of presheaves:

\[\cR^{\delta}[-]:\Psh(\cC)\rightarrow \Mod_{\cR}(\cC)\] and the right adjoint $(-)^{gp}$ is given section-wise. As a consequence Proposition \ref{Etgplike} extends to the categories of presheaves.

\begin{proposition}\label{ShvEtgplike}
Let $k$ an algebraically closed field and $C$ a presheaf of coalgebras. Then the unit of the adjunction $k^{\delta}[C^{gp}]\rightarrow C$ factors through the inclusion of presheaves $\acute{E}t(C)\subset C$ and induces a natural equivalence 
\[k^{\delta}[C^{gp}]\cong \acute{E}t(C).\]
\end{proposition}
\begin{proof}
This is trivial from \ref{Etgplike}
\end{proof}
\begin{theorem}
Let $k$ be a perfect field. Then for every presheaf of coalgebras $C$ the inclusion $\acute{E}t(C)\rightarrow C$ has a unique and natural splitting, which is a map of presheaves of coalgebras
\end{theorem}

\begin{proof}
This is a direct consequence of Theorem \ref{etalesplit}. Since the splitting is natural, this defines a morphism of presheaves of coalgebras $C\rightarrow \acute{E}t(C)\rightarrow C$.
\end{proof}

\begin{corollary}
Let $k$ an algebraically closed field. Then the counit of the adjunction given by  $k[C^{gp}]\rightarrow C$ has a natural retraction.  
\end{corollary}
\begin{proof}
This follows from the previous theorem, noting that the isomorphism in \ref{Etgplike} is natural.\end{proof}

\subsection{Category of presheaves of coalgebras for the Day convolution product $\coCAlg_{\cR}^{Day}(\cC)$}
For this section we fix the following notation. Let $\cRR$ be a commutative ring  and $(\cC,\otimes_{\cC},1_\cC)$ a  small $\Mod_\cRR$-enriched  $\cRR$-linear symmetric monoidal category.

\subsection{Day convolution Product}\label{DayConv} 
 
 Consider  $\mbox{Psh}(\cC,\Mod_\cRR)$ the category of presheaves, there is a natural extension of  the symmetric monoidal structure on $\cC$ to the  category of presheaves $\mbox{Psh}(\cC,\Mod_\cRR)$,  introduced by Day in \cite{MR0272852}. The  Day convolution product of  two presheaves $F$ and $G$  is given by the coend formula.

\[F\otimes^{Day} G=\int^{X, Y} \cC(-,X\otimes_\cC Y)\otimes_\cR F(X)\otimes_\cRR G(Y)\] in other words it is the left Kan extension of the external tensor product of $F$ and $G$, denoted by $F\overline{\otimes}G$,   along  $-\otimes_\cC -$. Explicitly it is the coequalizer of the two maps: 
\small
\begin{equation}\label{coeq}
\bigoplus_{\substack{(X,Y)\in \cC\times\cC\\
(X',X')\in \cC\times\cC}}\cC(U,X'\otimes Y')\otimes_\cRR  \cC\times\cC((X',Y'),(X,Y))\otimes_\cRR F(X)\otimes_\cRR G(Y)\rightrightarrows 
 \bigoplus_{(X,Y)\in \cC\times \cC}\cC(U,X\otimes Y)\otimes_\cRR F(X)\otimes_\cRR G(X) \end{equation} \normalsize  given  $\phi\in\cC(U,X'\otimes Y')$, $(\alpha,\beta)\in \cC\times\cC((X',Y'),(X,Y))$ and $s\otimes t\in F(X)\otimes_\cRR G(X)$, then the top map is given by:
 
\[ \phi\otimes (\alpha,\beta)\otimes (s\otimes t)\mapsto ((\alpha,\beta)\circ \phi)\otimes (s\otimes t) \] and the bottom map is given by:

\[ \phi\otimes (\alpha,\beta)\otimes (s\otimes t)\mapsto \phi\otimes (\alpha^{*}(s)\otimes \beta^{*}(t))) \] equivalently 
\[(F\otimes^{Day} G)(U)=\colim_{U\rightarrow X\otimes Y}F(X)\otimes G(Y)\] in particular for two representable sheaves $\mathbf{h}_X$ and $\mathbf{h}_Y$ the Day convolution product is given by the representable sheaf $\mathbf{h}_{X\otimes Y}$. In other words  the Yoneda embedding $y:\cC^{op}\rightarrow \cD$ is a symmetric monoidal functor. 

\begin{proposition}\label{Dayinternal}
Let $\cRR$ be a commutative ring  and $(\cC,\otimes_{\cC},1_\cC)$ a small $\Mod_\cRR$-enriched  $\cRR$-linear symmetric monoidal  category.   The monoidal category with the Day convolution product $(\mbox{Psh}(\cC,\mcM), \otimes_{Day}, h_{1_\cC})$ is a closed symmetric monoidal category. The internal hom is given by the end formula: 
\[[F,G]_{Day}=\int_{X,Y}Hom_{\mcM}(Hom_{\cC}(-\otimes_\cC X,Y),Hom_{\mcM}(F(X),G(Y))\]
\end{proposition}
\begin{proof}
\cite[Proposition 4.1]{MR862873} 
\end{proof}

As a consequence the category of presheaves with the Day convolution product is  monoidal and  cocomplete, i.e. all the endofunctors $F\otimes^{Day}-$, $-\otimes^{Day} G$ for $F,G\in \mbox{Psh}(\cC,\mcM)$ are cocontinuous. In  \cite{MR862873}  it is observed  that the monoidal structure given by the Day convolution is the free monoidal cocompletion of $\cC$, in the sense that:

\begin{proposition}\label{Universality}
Let $\cD$ be a monoidal cocomplete category and assume the condition in Proposition \ref{Dayinternal}. Then the functor  $[\mbox{Psh}(\cC,\mcM), \cD]\rightarrow [\cC,\cD]$ given by the composition with the Yoneda embedding induces an equivalence of categories between the cocontinuous monoidal functors $\Phi:\mbox{Psh}(\cC,\mcM)\rightarrow \cD$ and the monoidal functors $\phi:\cC\rightarrow \cD$. This equivalence  restricts to the corresponding subcategories of strong monoidal functors. Furthermore the strong monoidal functor $\phi:\mbox{Psh}(\cC,\mcM)\mbox{Psh}(\cC,\mcM)\rightarrow \cD$ are exactly the monoidal functors which are left adjoint.
\end{proposition}

\begin{proof}
{\cite[Proposition 5.1]{MR862873}}
\end{proof}

\subsection{Local  presentability of $\coCAlg_{\cR}^{Day}(\cC)$ }

Let $\cR$ be a field and $(\cC,\otimes_{\cC},1_\cC)$ a  small $\Mod_\cR$-enriched  $\cR$-linear symmetric monoidal category.

We extend the theorem  \ref{pshcofreesw} to the category of presheaves of coalgebras for the Day convolution product. The proof in \ref{cofree}  relies on the duality of finite dimensional vector spaces and the proof uses the discrete category $\cC_0$ associated to $\cC$, but the discrete category is not monoidal anymore. Here we have to work a bit more in the argument, which is inspired on the argument in \cite{BARR1974600} (see Theorems [3.1, 3.2] in \it{loc.cit}  \rm ).

\begin{lemma}\label{spsh}
Let $F\in \Mod_k(\cC)$ be a sheaf of $k$-vector spaces and suppose that we are given for every object $x\in Ob(\cC)$ a subspace $G_0(x)\subset F(x)$. There exists a  subpresheaf $G'$ of vector spaces, such that for every object $x\in Ob(\cC)$ $G_0(x)\subset G'(x)$ and  $\#(G')\leq\{\#(G),\aleph_0\}$.
\end{lemma}
\begin{proof}
By induction we define a collection of subspaces $G_n(x)\subset G(x)$ for every $x\in Ob(\cC)$ . For every $y\in\cC$  and $f\in\coprod_{x\in\cC}\mbox{Hom}_\cC(y,x)$ define the subspace $G_1(y)$ generated by $\langle f^{*}(G_0(x)), G_0(y)\rangle$, by definition $G_0(y)\subset G_1(y)$ but the restrictions maps $f^{*}$ are not yet defined. Assume $G_n(x)$ is given then  define $G_{n+1}$ as $G_{n+1}(x):=\langle f^{*}(G_n(x)),G_n(y)\rangle$  and take  $G':=\colim_{n}G_n(x)$, which is the desired sub-presheaf. 
\end{proof}

\begin{proposition}\label{pure}
Let $M$ and $N$ be  two presheaves of vector spaces and $M_0\subset M$ sub-presheaf of vector spaces. Then there exists $M'$ a presheaf of vector spaces such that $M_0\subset M'$  and $M'\otimes^{Day} N\rightarrow M\otimes^{Day} N$ is a monomorphism. Furthermore $\#(M')\leq\{\#(M_0), \, \#(\cC),\,\aleph_0\}$.
\end{proposition}
\begin{proof}
For simplicity, we write $-\otimes-$ instead of $-\otimes^{Day}-$.   Let $G_0=ker(M_0\otimes N\rightarrow M\otimes N)$  be the kernel of the canonical map. Then for each $x\in\cC$  consider $\{m_{i,x}\}$ a basis for $G_0(x)$. There exists a finite number of objects $(u_{i,j,x},v_{i,j,x})\in Ob(\cC\times\cC)$ and there are maps $x\rightarrow u_{i,j,x}\otimes v_{i,j,x}$ such that $m_{i,x}=\sum_{i,j,x} r_{i,j,x}\otimes s_{i,j,x}$ with $r_{i,j,x}\in M_0(u_{i,j,x})$ and $s_{i,j,x}\in N(v_{i,j,x})$. Since $m_{i,x}$ goes to zero in $(M\otimes N)(x)$, there exists a  finite number of maps $(f_{i,j,x},g_{i,j,x})\in Mor (\cC\times\cC)$ such that each pair fits in a commutative  diagram of the form

\[
\xymatrix{
x\ar[r]^-\alpha\ar[rd]_-{\alpha'}  &u'\otimes v'\ar[d]^-{f\otimes g}\\
&u\otimes v}.
\] and

\begin{equation}\label{Purecondition}
\sum_{i,j,x} (f^{*}_{i,j,x}\otimes g^{*}_{i,j,x})(r_{i,j,x}\otimes s_{i,j,x})=0
\end{equation}

Applying Lemma \ref{spsh} we can construct a subpresheaf $M_1$  such that for each $x\in Ob(\cC)$,  $M_1(x)$ contains the subspace generated by 

\[M_0(x) \coprod( \coprod_{x=u_{i,j,y}\atop x=v_{i,j,z}} \{r_{i,j,y},s_{i,j,z}\}),\] \it{i.e} \rm  if $x$ happen to be equal to $u_{i,j,y}$ for some $y\in\cC$ or $v_{i,j,z}$ for some $z\in\cC$.

By construction $G_0(x)$ goes to zero in $(M_1\otimes N)(x)$.  Assume that $M_{n-1}$ is constructed then we repeat the argument. We construct $M_n$  taking $G_{n-1}(x)=ker(M_{n-1}\otimes N)(x)\rightarrow (M\otimes N)(x) $ and the composition is equal to zero
\[\xymatrix{
G_{n-1}(x)\ar@{^{(}->}[r] \ar[rd] &(M_{n-1}\otimes N)(x)  \ar[d] \\
&(M_n\otimes N)(x)}
\]  we get a chain of subpresheaves $M_0\subset M_1\subset\dots\subset M_n$. Set $M:=\colim_{n}M_n$, this  is the desired presheaf.

\end{proof}

\begin{definition}
Let $M$ be  two presheaf of vector spaces and $M'\subset M$ a sub-presheaf of vector spaces. We say that $M'$ is a \it{pure sub-presheaf}    \rm   if  $M'\otimes^{Day} N\rightarrow M\otimes^{Day} N$ is a monomorphism
\end{definition}

The presheaf constructed in \ref{pure} is a pure sub-presheaf.

Let $F\in\mbox{ coCAlg}^{Day}_{K}(\cC)$ be a presheaf of coalgebras. A subpresheaf of vector spaces $M\subset F$ is called invariant if $\Delta_F(M)\subset Im(M\otimes^{Day} M\rightarrow F\otimes^{Day} F)$.

\begin{proposition}\label{invariant}
Given  $F\in\mbox{ coCAlg}^{Day}_{K}(\cC)$ and  $M_0$ a subpresheaf of vector spaces. There exists a invariant  sub-presheaf $M'$ of vector spaces such  that  $M_0\subset M'$ and  $\#(M')\leq\{\#(M_0), \#(\cC),\aleph_0\}$.
 \end{proposition}
 
\begin{proof}
For each $x\in Ob(\cC)$ let us  fix a basis $\{e_{x,i}\}$ for the subspace $M_0(x)\subset F(x)$. For each $e_{x,i}$ there exists a finite number of object $u_{x,i,j}, v_{x,i,j}\in\cC$ and a finite number of elements $m_{u,i,j}\in F(u_{x,i,j})$ and $n_{u,i,j}\in F(v_{x,i,j})$  such that $\Delta_F(F(x))\in (F\otimes F)(x)$ is contained in the subspace generated by  $\{m_{u,i,j}\otimes n_{v,i,j}\}$. Applying Lemma \ref{spsh} we construct $M_1$ the subsheaf of $F$ generated by $M_0(x)$ and $\coprod_{i,j,x} \{m_{i,j,x},n_{i,j,x}\}$ for all $x \in Ob(\cC)$. By construction $\Delta_F(M_0)(x)\subset(M_1\otimes M_1)(x)$, and $F_0$ is a subsheaf of  $M_1$, inductively we construct  $M_n$ such that $\Delta_F(M_{n-1})(x)\subset(M_n\otimes M_n)(x)$ and taking  $M'=\colim_{n}M_n$ we get a presheaf which is invariant under the diagonal map. 

\end{proof}

\begin{theorem}\label{lpstronggen}
Let $F\in\mbox{ coCAlg}^{Day}_{R}(\cC)$ and $M_0$ a presheaf of vector spaces of $\cC$. Then there exists a subcoalgebra $F'$ such that $M_0\subset F'$  and $\#(F')\leq\{\#(M_0), \#(\cC),\aleph_0\}$
\end{theorem}
\begin{proof}
Let  $F_1$ be the pure sub-presheaf of vector spaces given by Lemma \ref{pure} which contains $F_0$ and let $F_2$ be the invariant sub-presheaf of vector spaces which contains $F_1$, iterating both lemmas  $F_{n}$ is defined as the pure sub-presheaf associated to $F_{n-1}$ when $n$ is odd and $F_{n}$ the invariant sub-presheaf associated  $F_{n-1}$ when $n$ is even. Set $F'=\colim_{n}F_n$, it is clear that $F'$ is a pure and invariant. It remains to show that it is a subcoalgebra. By Lemma \ref{pure} $F'\otimes F'\rightarrow F\otimes F$ is injective. Thus the comultiplication map is induced by the comultiplication on $F$ and $F'\otimes F'\otimes F'\rightarrow F\otimes F\otimes F$ is injective, which give us the coassociativity. A similar observation is enough for the cocommutativity. The counit map is just the composition $F'\hookrightarrow F\xrightarrow{\varepsilon} k$.
\end{proof}

\begin{corollary}\label{lpres}
Let $\cC$ be a small category. The category of coalgebras  $\mbox{ coCAlg}^{Day}_{k}(\cC)$ is locally presentable, with strong generators given by $\{F\in \mbox{ coCAlg}^{Day}_{k}(\cC)  : \#(F)\leq \mbox{max}(\#(C),\aleph_0 )\}$.

\end{corollary}
\begin{proof}
Let $F_1 \overset{\eta}{\underset{\psi}{\rightrightarrows}} F_2$ be the maps of presheaves of coalgebras with $\eta\neq \psi$. Then there is some $x\in Ob(C)$ and $s\in F(x)$ with $\eta(x)\neq \psi(x)$. Let $M(x)$ be the presheaf vector spaces associated to the subspace of dimension one $\langle s\rangle$  then by the previous theorem exits $F_0\subset F_1$ a presheaf of coalgebras such that $M(x)\subset F_0(x)$ then the restriction of $\eta$ and $\psi$ to $F_0$ are to different maps. The generators are strong because for ever $K\in \mbox{ coCAlg}^{Day}_{K}(\cC)$  and $L\subset K$   a proper subpresheaf  there exists $x\in Ob(\cC)$ such that $L(x)\subset K(x)$ is a proper subspace. Let $s\in K(x)\setminus L(x)$ then take $G_s\subset K$ the subsheaf associated to the vector space generated by $s$ and this map does not factorize through $K$.
\end{proof}

\begin{theorem}\label{DayFTC}
The underlying functor $U:\coCAlg^{Day}(\cC)\rightarrow \Mod_{k}(\cC)$ has a right adjoint and is comonadic.
\end{theorem}

\begin{proof}
The proof follows word by word \cite[Theorem 4.1]{BarrMichael1985Ttat}
\end{proof}

\begin{rmk}
The category of coalgebras is cartesian closed, with the product given by the Day convolution product of two coalgebras. 
\end{rmk}

We would like to get a more refined description  of category  $\coCAlg^{Day}(\cC)$. Since the coalgebra objects in this category are not section-wise coalgebras, we can not extend the results from Section \ref{coalgclassif}  easily. As well, the existence of a left adjoint functor $\Psh(\cC')\rightarrow \coCAlg^{Day}_{\cR}(\cC)$  is not totally straight forward.  Nevertheless, if we restrict to the category of coalgebra objects in $\PST$ (the category of presheaves with transfers of Voevodsky) we are able to get a nice description. We will discuss these problem in a future paper.

\section{Recollections of  Motivic Spaces}\label{MotSp}

The theory of motivic spaces and unstable motivic homotopy theory was set up by Morel and Voevodsky in \cite{MorelVoevodsky}. Roughly speaking a motivic space (resp. an \'etale motivic) is a simplicial presheaf which is $\mathbb{A}^1$-homotopy invariant and satisfies Nisnevich descent (resp. \'etale descent). This has been widely discussed in the literature. In the following section we recall the local and global model structures for simplicial presheaves  studied in \cite{MR906403} and \cite{MorelVoevodsky}.  Furthermore, we state results by Dugger \cite{MR2034012} and , which allows to see Jardine's model structure as a left Bousfield localization the local objects in Jardine's model structures satisfy hyperdescent.


\subsection{Recollection in motivic model structures}


\subsection{Generalities and basic definitions}
Denote by  $\sS$ the category of simplicial sets endowed with the  Kan model structure; where a map $f$ is a weak equivalence if the induced  map of geometric realizations $|f|:|X|\rightarrow|Y|$ is an homotopy equivalence of topological spaces, it is a cofibrations if it is a monomorphism levelwise and is a fibrations if it is a Kan fibration (see \cite[Example A.2.7.3]{lurie2009higher})

Before introduce the motivic model structure, let us give the definition of right and left induced model structures. Left induced model structures are particularly interesting for us in order to induce a model structure in the category of coalgebras.

\subsubsection{Right and left induced model structures}

Let $\cM$ and $\cN$ be complete and cocomplete categories. Assume that either $\cM$ or $\cN$ is a model category and let $\xymatrix{\cM\ar@<0.5ex>[r]^{L} &\cN \ar@<0.5ex>[l]^{R}}$  be a pair of adjoint functors. A standard question is wether we can use $R$ or $L$ to build a model structure on $\cN$ or $\cM$. More specifically we can introduce the notion of \it{left} \rm or \it{right} \rm  \it{induced model category}.\rm

\begin{definition}
Let $\xymatrix{\cM\ar@<0.5ex>[r]^{L} &\cN \ar@<0.5ex>[l]^{R}}$ be an adjoint pair of functors.

\begin{enumerate}
\item Assume that $\cM$ is a model category and  $\cN$ is  a complete and cocomplete category. If the the classes of morphisms  $R^{-1}(Fib)$ and $R^{-1}(W)$ satisfy the axioms of a model category. Then it is called the \it{right-induced model structure}.\rm 

\item Assume that $\cN$ is a model category and  $\cM$ is  a complete and cocomplete category. If the the classes of morphisms  $L^{-1}(Cof)$ and $L^{-1}(W)$ satisfy the axioms of a model category. Then it is called the \it{left-induced model structure}.\rm 
\end{enumerate}
\end{definition}

\begin{example}

The category of simplicial $\cR$-modules $\sMod_\cR$ is endowed with the  right-induced model structure  from the Kan model structure on $\sS$ by the forgetful functor $U:\sMod_\cR\rightarrow \sS$.       
\end{example}

In \ref{leftcoal} we are going to study an example of left induced model structure, the homotopy category of coalgebras.

Via the Dold-Kan equivalence   $N:\sMod_\cR\rightleftarrows Cplx(\cR)_{\geq 0}:\Gamma$ we give $Cplx(\cR)_{\geq 0}$ a model structure; in this model structure, the weak equivalences are the quasi-isomorphisms.

Let $\cC$ be a small category, by Theorem \ref{proinj} it is possible to endow the category $s\Psh(\cC)$ and $\sMod_{\cR}(\cC)$ with the injective and the projective model structures. Moreover, we get the following commutative diagram.

\[\xymatrix{s\Psh(\cC)_{proj} \ar@<0.5ex>[r]^{\cR} \ar@<-0.5ex>[d] & \sMod_\cR(\cC)_{proj}\ar@<0.5ex>[l]^{U}\ar@<-0.5ex>[d]\\
s\Psh(\cC)_{inj} \ar@<0.5ex>[r]^{\cR}\ar@<-0.5ex>[u]&\sMod_{\cR}(\cC)_{inj}.\ar@<0.5ex>[l]^{U} \ar@<-0.5ex>[u]}\]

Since $\sS$ and  $s\Mod_{\cRR}$ are combinatorial model categories \cite[A.2.7]{lurie2009higher}, both the \it projective \rm and \it injective \rm model structures over $s\mbox{Mod}_\cRR(\cC)$ are combinatorial model structures and then, by definition, cofibrantly generated.  For the projective model structure, a set of generators is given by

\begin{align*}
\cI&=\{\cRR[id_X]\otimes \cRR[\iota_n]:\cRR[h_X]\otimes\cRR[\partial\Delta^{n}]\rightarrow \cRR[h_X]\otimes \cRR[\Delta^n],\,\,\, for \,n\geq 0,\, X\in \cC\}\\
\cJ&=\{\cRR[id_X]\otimes \cRR[j_k^n]:\cRR[h_X]\otimes\cRR[\partial\Lambda^{n}_k]\rightarrow \cRR[h_X]\otimes \cRR[\Delta^n],\,\,\, for \,n\geq 1,\,\ 0\leq k\leq n\,\ X\in \cC\}
\end{align*} from this description immediately follows that every representable simplicial presheaf is cofibrant for the projective model structure. 

The \it injective model structure \rm is also cofibrantly generated, we would like to make some remarks about the generators since we are particularly interested in this model structure. The proof of Theorem \ref{proinj} follows applying  Theorem \ref{GenModel}. Since $\sMod_\cRR(\cM)$ is a Grothendieck abelian category by \cite{MR1780498} the class of monomorphisms is cofibrantly generated by the set of monomorphisms 

\[\cI=\{f:A\rightarrow B : \,\, B\,\, \mbox{is} \,\, \kappa\mbox{-bounded}\}.\]
Since the categories of simplicial presheaves and presheaves of simplicial modules are presentable categories, the set
\[\cJ=\{f:A\rightarrow B : \,\, B\,\, \mbox{is} \,\, \kappa\mbox{-bounded and } f\mbox{ is section-wise weak equivalence}\}
\] generates the trivial cofibrations. 

This fact is also observed in the axioms given by Jardine, for any monomorphism $f:X\rightarrow Y$ which is weak equivalence, and  any $A$ $\kappa$-bounded, there exists $B$ $\kappa$-bounded subobject $B\subset Y$ such that $B\cup X\rightarrow B$ is a weak equivalence. 
    
\[
\begin{tikzcd}[row sep=0.7em]
 &          &X\arrow[dd]\\
 & X\cap B \arrow[ru] &\\
A\arrow[rr]\arrow[dr] &           & Y\\
& B \arrow[ru]  \arrow[uu,<-,crossing over]& 
\end{tikzcd}
\]
                       
\subsection{Local model structures}

Let  $(\cC,\tau)$ be a small Grothendieck site. Following Jardine (see \cite{MR906403}) we can study local model structures on the category of simplicial presheaves $s\Psh(\cC,\tau)$. We denote $\pi_{n}(\cX)$ the $n$-homotopy presheaf of $\cX$. For a presheaf $F$ on $(\cC,\tau)$, we write $\widetilde{F}$ for the associated sheaf.

\begin{definition}
A map $\cX\rightarrow \cY$ of simplicial presheaves is called a local weak equivalence if 

\begin{enumerate}
\item the map $\widetilde{\pi_0}\rightarrow \widetilde{\pi_0}$ is an isomorphism of sheaves/

\item for any object $U$ of $\cC$, any $x\in\cX(U)$ and any $n>1$ the morphism of associated sheaves 
\[\widetilde{\pi_n}(\cX,x)\rightarrow \widetilde{\pi_n}(\cY, f(x))\]
on the overcategory $\cC/U$ defined by $f$ is an isomorphism.
\end{enumerate}
\end{definition}

\begin{definition}\label{stalkweak}
Let  $(\cC,\tau)$ be a small Grothendieck site and $f:\cX\rightarrow \cY$ a morphism of simplicial presheaves:
\begin{enumerate}
\item $f$ is called a simplicial weak equivalence if for any point $x$ of the site $(\cC,\tau)$ the morphism of simplicial sets $x^{*}(f):x^{*}(\cX)\rightarrow y^{*}(\cX)$ is a weak equivalence.
\item $f$ is called a cofibration if it is a monomorphism;
\item $f$ is called a fibration if it has the right lifting property with respect to any cofibration which is a weak equivalence.  
\end{enumerate}
Denote by $\mathbf{W}_s$ (resp. $\mathbf{C}$, $\mathbf{F}_s$) the class of simplicial weak equivalences (resp. cofibrations, simplicial fibrations). 
\end{definition}

\begin{rmk}
If $(\cC,\tau)$ is a small Gorthendieck site with enough points. Then the notion of simplicial weak equivalence and local weak equivalence coincides (see \cite[Section 2]{MR906403}).
\end{rmk}

\begin{theorem}\label{Jardinemodel}
Let  $(\cC,\tau)$ be a small Grothendieck site with  enough points. Then the triple $(\mathbf{W}_s, \mathbf{C}, \mathbf{F}_s)$ defines a model structure in the category   $s\Psh(\cC)$.
\end{theorem}
\begin{proof}
This is proved in \cite[Theorem 2.3]{MR906403}.\end{proof}
This model structure is known as Jardine model structure.

In general it is hard to describe the fibrations or give an explicit fibrant replacement. In  \cite{MorelVoevodsky} they construct an explicit resolution fibrant replacement. Before describe this  fibrant replacement lets provided the following definition of local fibration.

\begin{definition}\label{localfib}
A morphism of simplicial presheaves $f:\cX\rightarrow \cY$ is called a \it local fibration \rm (\it resp. \rm trivial local fibration) if for any point $x$ of $\tau$ the corresponding morphism of simplicial sets $x^{*}(\cX)\rightarrow y^{*}(\cY)$ is a Kan fibration (\it resp. \rm a Kan fibration and a weak equivalence).  A morphism of simplicial presheaves $f:\cX\rightarrow \cY$ is called \it local cofibration \rm if it satisfies the left lifting property with respect to all trivial fibrations.

\end{definition}
\begin{rmk}\label{locfib}
Every simplicial fibrations is a local fibration, but not every local fibration is a simplicial fibration. Eilenberg-Mac Lane spaces fails to be simplicial fibrant. In \cite{MR906403} it is called the global and local theory respectively. 
\end{rmk}

\subsection{Hyperdescent and \v Cech descent}\label{hypchdes}

Note that class of section-wise weak equivalences $\mathbf{W}$ is contained in the class of simplicial weak equivalences $\mathbf{W}_s$ and the cofibrations are the injective cofibrations. As a consequence Jardine model structure is a left Bousfield localization of the injective model  structure.  The local objects are the simplicial presheaves satisfying $\tau$-hyperdescent \rm i.e. descent for arbitrary hypercovers. Recall that a hypercover of $X$ is a simplicial presheaf with an augmentation $U\rightarrow X$, such that each $U_n$ is a coproduct of representable, and $U\rightarrow X$ is a local acyclic fibration.   

\begin{definition}
An object-fibrant simplicial presheaf $F$ satisfies descent for a hypercover $U\rightarrow X$ if the natural map from $F(X)$ to the homotopy limit of the diagram
\[\xymatrix{\prod_{a}F(U_0^a)\ar@<0.5ex>[r]\ar@<-0.5ex>[r]& \prod_{a}F(U_1^a) \ar@<1ex>[r]\ar[r]\ar@<-1ex>[r]&\cdots}\] is a weak equivalence, where the product runs over all the representable summands of each $U_n$. If $F$ is not objectwise-fibrant we say the $F$ satisfies   descent if some object-wise fibrant replacement satisfies descent.
\end{definition}

\begin{theorem}
Let $(\cC,\tau)$ be an small Grothendieck site and $S$ the collection of hypercovers. Then the Bousfield localization at $S$ exists and coincides with the Jardine's model structure. We denote this category as $s\Psh(\cC)_{\tau, inj}$.
\end{theorem}
\begin{proof}
This is proved in \cite[Corollary 7.1]{MR2034012}
\end{proof}

\begin{definition} Let $(\cC,\tau)$ be a small Grothendieck site and  $S_\tau$ the set of covering sieves for $\tau$, which is given by the set of monomorphisms $R\hookrightarrow X$ with $X$ the representable presheaves.  We say that a simplicial presheaf on $\cC$ satisfies $\tau$-\v Cech descent  or is $\tau$-\v Cech local or $\tau$- \v Cech fibrant if it is $\cS_\tau$-local  and we say that a morphism of simplicial presheaves is a $\tau$-local weak equivalence if it is an $\cS_\tau$-equivalence.   
\end{definition}

\begin{rmk}
Since $(\cC,\tau)$ is an small Grothendieck site,  $ S_\tau$ is actually a set. And since $s\Psh(\cC)_{inj}$ is a combinatorial model structure, the Bousfield localization with respect to $ S_\tau$ exists. We denote this model structure as $s\Psh(\cC)_{\check{\cC},\tau, inj}$
\end{rmk}

If $\mathcal{U}=\{U_i\rightarrow X\}_{i\in I}$ is a family of maps in $\cC$, we denote by $\check{C}(\mathcal{U})$ its \v Cech nerve, which is a simplicial presheaf on $\cC$ with an augmentation to $h_X$. The following result gives a characterization of $\tau$-local objects in terms of \v Cech descent. 

\begin{lemma}
Let $(\cC,\tau)$ be a small Grothendieck site. For every $X\in \cC$, let $Cov(X)$ be a set whose elements are families of maps $\{U_i\rightarrow X\}_{i\in I}$ which are coverings for $X$ in the Grothendieck topology $\tau$.  A simplicial presheaf $F$ on $\cC$ is $\tau$-local if and only if it is $\cS$-local, \it{i.e.} if and only if for every $X\in \cC$ and every $\{ U_i\rightarrow X\}_{i\in I}\in Cov(X)$, the canonical map 

\[ F(X)\rightarrow \mbox{holim}_{n\in\Delta}\prod_{i_0,\dots, i_n\in I}F(U_{i_0\times_X\cdots\times_XU_{i_n}})\] is a weak equivalence.   
\end{lemma}  
\begin{proof}
\cite[Lemma 3.1.3]{asok2017}
\end{proof}

In general the Jardine model structure $s\Psh(\cC)_{\tau,inj}$ is a left Bousfield localization of the \v Cech $\tau$-local model structure $s\Psh(\cC)_{\check{\cC},\tau, inj}$. Since we are inverting arbitrary hypercovers and not just \v Cech hypercovers, and the two model structures may be different, as it is exhibited in \cite[Example A.9]{MR2034012}. However, if $S$ is a Noetherian scheme of finite type and finite Krull dimension and $Sm_S$ is provided with the Nisnevich topology both localizations coincide \cite[Example A.10]{MR2034012}. 


\subsection{Nisnevich and \'etale-descent}

From now we will be particularly interested in the standard topologies for algebraic geometry, the  Zariski, Nisnevich and the \'etale topology.

\begin{definition}
Let $\cC$ be a small category with initial object $0$.
A \it{cd-structure} \rm$P$  on $\cC$ is a collection of commutative squares in $\cC$. We say that $P$ forms a cd-structure on $\cC$ if whenever $\mathcal{Q}\in P$ and $\mathcal{Q}'$ is isomorphic to $\mathcal{Q}$, then $\mathcal{Q}'$ is also in $P$. The squares of the collection $P$ are called distinguished squares.
\end{definition}

\begin{definition}
The \it cd-topology,  \rm  $\tau_p$, associated to a \it{cd-structure }\rm$P$ is the Grothendieck topology on $\cC$ generated by the coverings sieves of the following form:
\begin{enumerate}
\item The empty sieve is a covering sieve of the initial object $0$,
\item The sieve generated by morphisms of the form $\{A\rightarrow X, Y\rightarrow X\}$ where $A\rightarrow X$ and $Y\rightarrow X$ are two sides of the squares in $P$ of the form 
\[\xymatrix{B\ar[d]\ar[r]&Y\ar[d]_p\\
                  A\ar[r]_e&X}\]
\end{enumerate} 
\end{definition}

\begin{example}
\leavevmode
\begin{enumerate}
\item The \it Zariski  $cd$-structure \rm on $Sm_S$ consisting of Cartesian squares
\[\xymatrix{U\times_XV\ar[d]\ar[r]&V\ar[d]_p\\
                  U\ar[r]_e&X}\] where $U, V\subset X$ are open subschemes such that $U\cap V=X$.
\item The \it{Nisnevich} cd-structure \rm on $Sm_S$ consists of Cartesian squares 
\[\xymatrix{U\times_XV\ar[d]\ar[r]&V\ar[d]^\pi\\
                  U\ar[r]^j&X}\]  where $j$ is an open immersion and $\pi$ is \'etale and furthermore $\pi$ induces an isomorphism $V\times_XZ\cong Z$, where $Z$ is the reduced closed complement of $j$.
\end{enumerate}
\end{example}

\begin{rmk}
Since the schemes on $Sm_S$ are quasi-compact and quasi-separated, the topology generated by the Zariski \it cd-structure \rm is the usual Zariski topology. The proof in  \cite[Proposition 1.4]{MorelVoevodsky} shows that the topology generated by the \it Nisnevich cd-structure \rm is the Nisnevich topology.
This fact is not true for the \it  \'etale topology \rm but in \cite{ISAKSEN200437} they discuss the behavior of the \'etale topology in terms of distinguished squares. 
\end{rmk}

\subsection{Excision and descent}

\begin{definition}
Let $\cC$ be a small category with an initial object $0$ and let $P$ be a \it cd-structure \rm on $\cC$. A simplicial presheaves $F$ on $\cC$ satisfies $P$-excision if 
\begin{enumerate}
\item F(0) is weakly contractible;
\item for every square $Q$ in $P$, F(Q) is homotopy Cartesian.
\end{enumerate} 
\end{definition}

\begin{proposition}
Let $\cC$ be a small category with a strictly initial object let $P$ be a \it cd-structure \rm on $\cC$ such that
\begin{enumerate}
\item Every square in $P$ is Cartesian;
\item pullbacks of squares in $P$ exists and belongs to $P$;
\item for every square in $P$ $e:A\rightarrow X$ is a monomorphism;
\item for every square in $P$, the square 
\[\xymatrix{W\ar[d]\ar[r]&V\ar[d]_p\\
                  W\times_UW\ar[r]_e&V\times_XV}\] is also in $P$.

\end{enumerate} Then a simplicial presheaves $F$ on $\cC$ satisfies $P$-excision if and only if satisfies $\tau$-\v Cech descent. 
\end{proposition}


\subsection{Derived categories of presheaves of $\cRR$-modules}         
\begin{definition}
A morphism $F\rightarrow G$ in $s\Mod_{\cRR}(\cC)$ is called $\tau$-local weak equivalence if for any point of the site $(\cC,\tau)$ the morphism of simplicial sets $x^{*}(f):x^{*}(F)\rightarrow x^{*}(G)$ is a weak equivalence in $s\Mod_{\cRR}$.
\end{definition}

\begin{theorem}\label{mono}
Let $\cRR$ be a commutative ring and $(\cC,\tau)$ a small Groethendieck site with enough points. There is a simplicial, cofibrantly generated model category structure on $s\Mod_{\cRR}(\cC)$ where the cofibrations are the monomorphisms and the weak equivalences are object-wise weak equivalences (resp. $\tau$-local weak equivalence) of the  underlying simplicial sets. 
\end{theorem}

 \begin{proof}
 The proof follows as \cite[Theorem 5.7]{MR3073905}. Let us denote $W_{\cRR}$ (resp. $W_{\tau,\cRR}$) the class of object-wise weak equivalences (resp. $\tau$-local weak equivalence) in $s\Mod_{\cRR}(\cC)$. The forgetful functor $U:s\Mod_{\cRR}(\cC)\rightarrow s\Psh(\cC)$ is a right adjoint hence preserves filtered colimits, as a consequence it is an accessible functor. Then the condition \it (3) \rm from the theorem \ref{GenModel} is satisfied. Condition \it (4) \rm is trivial. 
 
 Since $s\Mod_{\cRR}(\cC)$ is a Grothendieck abelian category by \cite{MR1780498} the class of monomorphisms is cofibrantly generated by a set of monomorphisms. Thus condition \it(1) \rm is satisfied. It remains to show conditions \it (2) \rm and \it (5). \rm 
 
 For \it (2) \rm  the class of monomorphisms which are object-wise weak equivalences (resp. $ \cW_{\tau,\cRR}$) are closed under  transfinite composition and push-outs.  The closure under transfinite composition follows because it is true in $s\Psh(\cC)$, and it transfinite composition can be computed in $s\Psh(\cC)$. 
 
 So is enough to show that for every push-out square:

\[
\xymatrix{A\ar[r]\ar[d]^-j&X\ar[d]^-f\\
                   B\ar[r]&Y}
\]                 
                   
 where $j\in (mono)\cap \cW_{\cRR}$ (resp. $(mono)\cap \cW_{\tau,\cRR}$) then   $f\in \cW_{\cRR}$ (resp. $\cW_{\tau,\cRR}$) . 

The case  $j\in (mono)\cap \cW_{\cRR}$ is obvious because  local weak equivalences in $\sS$ which are monomorphisms are closed under push-outs. 

\it For the case  $j\in(mono)\cap \cW_{\tau,\cRR}$\rm : Since push-outs are preserved under colimits we have the following push-out diagram

\begin{equation}\label{push-out}
\xymatrix{x^{*}(A)\ar[r]\ar[d]^-j&x^{*}(F)\ar[d]^-f\\
                   x^{*}(B)\ar[r]&x^{*}(G)}
\end{equation}
$x^{*}(j):x^{*}(A)\rightarrow x^{*}(B)$  is a monomorphism and also a weak equivalence in $\sS$. Since local weak equivalences in $\sS$ which are monomorphisms are closed under push-outs.

Condition \it (5) \rm is verified because we can consider 
\[\cRR[\cI^{proj}]:=\{\cRR [h_U]\otimes \cRR[\partial\Delta_{n}]\longrightarrow\cRR [h_U]\otimes \cRR[\Delta_{n}] \,\ :\,\, U\in \cC, \, \, n\in\mathbb{N}\} \] the set of generating cofibrations on $s\Mod_\cRR(\cC)_{proj}$ which are also monomorphisms then $mono-inj\subset (\cRR_{tr}[\cI^{proj}])-inj$  but a morphism in the former set is also a section-wise weak equivalence and thus a simplicial weak equivalence $W_s$. We have therefore defined a cofibrantly generated model category.

\end{proof}

\begin{rmk}
The class of object-wise weak equivalences $\cW_{\cRR}$ and the class of monomorphisms  actually define the injective model structure in the category $Fun(\cC, s\Mod_\cRR)$. Furthermore, the model structure defined by the  class $\cW_{\tau,\cRR}$ and the class of monomorphisms is the Bousfield localization of the injective model structure with respect to the class 
\[\{\cRR[h_\cU]\rightarrow \cRR[h_X] \,\,:\,\, \cU\rightarrow X \mbox{ is an hypercover of X and } X\in \cC\}\] This follows because a morphism $A\rightarrow B$ is an element in $\cW_{\cRR}$ if and only of $x^{*}A\rightarrow x^{*}B$ is a weak equivalence of simplicial sets. Then by the Whitehead Theorem  the morphism on stalks  is a quasi-isomorphism in $s\Mod_\cRR$ for all points $x$. Then $A\rightarrow B$ induces an isomorphism in $a_\tau H_{*}(A)\rightarrow a_\tau H_{*}(B)$. 
 \end{rmk}
\subsection{$\mathbb{A}^1$-localization}
We now introduce the $\mathbb{A}^1$-localization. We will consider $Sm_F$ the category of smooth schemes of finite type over $F$ with a Grothendieck topology $\tau$, which is either the Nisnevich topology or the \'etale topology.
\begin{definition}
A presheaf of simplicial $\cRR$-modules $M$ is $\mathbb{A}^1$-local if for any $X\in Sm_F$ the induced map 
\[\mapmod(\cRR[h_X],M)\rightarrow \mapmod(\cRR[h_{X\times \mathbb{A}^1}],M)\] is a weak equivalence. A morphism $M\rightarrow N\in s\Mod_{\cRR}(Sm_F)$ is a $\mathbb{A}^1$-weak equivalence if for any $P\in s\Mod_{\cRR}(Sm_F)$ $\mathbb{A}^1$-local, the induced map  $\mapmod(N,P)\rightarrow \mapmod(M,P) $ is a weak equivalence.
Since $s\Mod_\cRR(Sm_F)$ is endowed with the local cofibrantly generated model category $s\Mod_{\cRR}(Sm_F)_{inj,\tau}$, then the Bousfield localization exists with respect to the set 
$$\{ \cRR[h_{X\times\mathbb{A}^1}]\rightarrow \cRR[h_X]\,\, : \,\, X\in Sm_F\}$$ we denote this model structure as $s\Mod_\cRR(Sm_F)^{\mathbb{A}^1}_{inj,\tau}$.
\end{definition}

Note that the model category $s\Mod_\cRR(Sm_F)^{\mathbb{A}^1}_{inj,\tau}$ has as cofibrations the monomorphisms since cofibrations do not change under Bousfield localization.

\begin{proposition}\label{free-forget}
If we endow the category $s\Psh(Sm_F)$ with the injective (resp. $\tau$-local and $\tau$-$\mathbb{A}^1$-inj) and $s\Mod_{\cRR}(Sm_F)$ with the injective (resp. $\tau$-inj and $\tau$-$\mathbb{A}^1$-inj) \ref{msCoalgebras}.
Then the adjunction (free-forgetful adjunction)
\[\cRR[-]:s\Psh(Sm_F)\rightleftarrows s\Mod_{\cRR}(Sm_F):u\] is a Quillen adjunction.
\end{proposition}

\begin{proof}

The cofibrations in the injective (resp. $\tau$-local and $\tau$-$\mathbb{A}^1$-local) model structure are the monomorphisms, which clearly are preserved by $\cRR[-]$
                                              
We claim that $\cRR[-]$  preserves trivial cofibrations in the injective (resp. $\tau$-local and $\tau$-$\mathbb{A}^1$-local) model structure. For the injective model structure, this follows because every sectionwise weak equivalence induces a sectionwise homology equivalence by the Whitehead Theorem. 

For the $\tau$-local model structure, let $x$ be a point in the Grothendieck site $(Sm_F)_\tau$. Since $\cRR[-]$ commutes with filtered colimits   $\cRR[x^{*}(\cX)]=x^{*}(\cRR[\cX])$, and then  by the Whitehead theorem again it follows that $\cRR[-]$ preserves $\tau$-local weak equivalences.

For the $\tau$-$\mathbb{A}^1$-local model structure, we claim that  $\cRR[-]$ preserves $\mathbb{A}^1$-weak-equivalences. To show this, let $N:\sMod(Sm_F)\rightarrow Cplx(Sm_F, \Mod_{\cRR})$ be the functor of taking the normalizing chain complex of simplicial presheaves of modules and let  $C_{*}(-,\cRR)=N\circ \cRR[-]$.

\[\xymatrix{s\Psh(Sm_F)\ar[r]^{\cRR[-]}\ar@/^2.0pc/[rr]^{C_{*}(-,\cRR)}&\sMod(Sm_F)\ar[r]<0.5ex>^-N&Cplx(Sm_F, \Mod_{\cRR})}\] 

Notice that $A_{*}\rightarrow B_{*}$ is an $\mathbb{A}^1$-weak equivalence in $\sMod(Sm_F)$ when $N(A_{*})\rightarrow N(B_{*})$ is an $\mathbb{A}^1$-weak equivalence of $Cplx(Sm_F, \Mod_{\cRR})$. We need to show that $C_{*}(-,\cRR)$ transforms $\mathbb{A}^1$-weak equivalences to $\mathbb{A}^1$-weak equivalences. Recall that $C_{*}(-,\cRR)$ has a right adjoint 
             \[K:Cplx(Sm_F, \Mod_{\cRR})\rightarrow s\Psh(Sm_F)\] called the Eilenberg-MacLane space functor.  If $C_{*}$ is an $\mathbb{A}^1$-local complex then $K(C_{*})$ is an $\mathbb{A}^{1}$-local space. We claim that  $C_{*}(-,\cRR)$ preserves $\mathbb{A}^1$-weak equivalence in $s\Psh(Sm_F)$. 
 Let $\cX\rightarrow \cY$ be an $\mathbb{A}^1$-weak equivalence and $C_{*}$ be an $\mathbb{A}^1$-local complex $C_{*}$, since $K(C_{*})$ is $\mathbb{A}^1$-local and by adjointness 
 \[\xymatrix{Hom(\cY, K(C_{*}))\ar[r]^{\cong}\ar[d]_{\cong}&Hom(\cX, K(C_{*})\ar[d]_{\cong}\\
                  Hom(C_{*}(\cY,\cRR),C_{*})\ar[r]&Hom(C_{*}(\cX,\cRR),C_{*})}\] then the claim follows and $\cRR[-]$ preserves $\mathbb{A}^1$-weak equivalences.


\end{proof}


\section{Homotopy theory for simplicial coalgebras and $\mathbb{A}^1$-Goerss Theorem for $k$-algebraically closed}\label{leftcoal}

Since $\mbox{coCAlg}_{R}(Sm_F)$ is complete and cocomplete by \cite[Theorem 2.5 ]{GoerssPaulG.1999Sht} the category $s\mbox{coCAlg}_{R}(Sm_F)$ is a simplicial category, where the simplicial structure is given by:   
 \[-\otimes-:\sS\times s\mbox{coCAlg}_\cRR(Sm_F)\rightarrow s\mbox{coCAlg}_\cRR(Sm_F)\] defined on sections by 
  \[(K\otimes F)_n(U):=\cRR[K_n]\otimes_\cRR F_n(U)\] for $F\in s\mbox{coCAlg}_\cRR(Sm_F)$ and $K\in \sS$. Giving $\phi:[n]\rightarrow[m]$  the induced map $\phi^{*}$ is giving by the composition
  \[\cRR[K_m]\otimes_\cRR F_m(U)\rightarrow\cRR[K_m]\otimes_\cRR F_n(U)\rightarrow \cRR[K_n]\otimes_\cRR F_n(U) \] and remember that the colimits on $\mbox{coCAlg}_\cRR(Sm_F)$ are created by the forgetful functor then the coalgebra structure on $K\otimes F$ is given by the tensor of the coalgebra structure on $F$ and the coalgebra structure in $\cRR[K]$ induced by the diagonal map $\Delta_K:K\rightarrow K\times K$.
  The mapping space functor is given by

\[\mbox{Map}_{s\mbox{coCAlg}_\cRR(Sm_F)}(F,G)_n:=\mbox{Hom}_{s\mbox{coCAlg}_\cRR(Sm_F)}(\cRR[\Delta^n]\otimes F,G).\] 
\\

Given the adjunction
\[ U:s\mbox{coCAlg}_{\cRR}(Sm_F)\rightleftarrows s\mbox{Mod}_{\cRR}(Sm_F): CF \] we would like to lift  a given model structure on $s\mbox{Mod}_{R}(Sm_F)$ along the right adjoint functor in order that the adjunction becomes a Quillen adjunction and that the category $s\mbox{ coCAlg}_{R}(Sm_F)$ satisfies the axioms of a simplicial model structure. This is dual to the usual Quillen's argument of transfer a model structure along a left adjoint, but close to the arguments provided for the existence of the Bousfield Localization. The strategy is essentially provided a class of weak equivalence and cofibrations and prove the existence of the model structure using Theorem \ref{GenModel}.

This proof is essentially given in \cite{MR3073905}. We reproduced the proof here for completeness. 
\begin{theorem}\label{ModelStrCoalgebras}\label{msCoalgebras}
Let $\cRR$ be a commutative ring and $s\Mod_\cRR\cRR(Sm_F)$  the category of simplicial modules endowed with the injective (resp. $\tau$-local and $\tau$-$\mathbb{A}^1$-inj). There is simplicial, cofibrantly generated model category structure on $s\mbox{\rm coCAlg}_{\cRR}(Sm_F)$ left induced by the forgetful functor
\[U:s\mbox{\rm coCAlg}_{\cRR}(Sm_F)\rightarrow s\Mod_\cRR(Sm_F)\] where the class of weak equivalences in $\cW_{\cRR}$  (resp. $\cW_{\cRR,\tau}$ and $\cW_{\cRR,\tau}^{\mathbb{A}^1}$) is given by the maps of coalgebras $f:C\rightarrow D$ such that $U(f): U(C)\rightarrow U(D)$ is an object-wise (resp. $\tau$, $\tau$-$\mathbb{A}^1$) weak equivalence  and the set of $U$-monomorphisms between $\kappa$-presentable objects is the generating set of cofibrations (for any choice of a large enough regular cardinal $\kappa$).

\end{theorem}
\begin{proof}
It is enough to show that the class of weak equivalences $\cW_\cRR$ (resp. $\cW_{\cRR,\tau}$, $\cW_{\cRR,\tau}^{\mathbb{A}^1}$)  and class of $U$-monomorphisms satisfies the conditions \it(1)-(5)   \rm  from Theorem \ref{GenModel}.  \rm   Condition \it (3) \rm is satisfied because $\cW_\cRR$ (resp. $\cW_{\cRR,\tau}$, $\cW_{\cRR,\tau}^{\mathbb{A}^1}$) is the inverse image of $\cW$ (resp. $\cW_s$,$\cW_{\tau}^{\mathbb{A}^1}$) under the composition $u\circ U$ where $U:s\mbox{\rm coCAlg}_{\cRR}(Sm_F)\rightarrow s\Mod_\cRR(Sm_F)$ and $u:s\Mod_\cRR(Sm_F)\rightarrow s\Psh(Sm_F)$. Since $U$ is a left adjoint it preserves all colimits and $u$ is a right adjoint then it preserves filtered colimits it follows that $u\circ U$ is an accessible functor. Then by proposition \ref{Modelinverse} the class of weak equivalences $\cW_\cRR$ (resp. $\cW_{\cRR,\tau}$, $\cW_{\cRR,\tau}^{\mathbb{A}^1}$) is accessible embedded in the category of arrows $Arrow(s\mbox{\rm coCAlg}_{\cRR}(Sm_F))$. Condition \it(4)   \rm   is obvious. Conditions \it (1)   \rm  and \it(2)   \rm  are satisfied because colimits in $s\mbox{\rm coCAlg}_{\cRR}(Sm_F)$ are created by the forgetful functor and Theorem \ref{mono}.
It remains to show condition \it(5) \rm, this will be a consequence of Lemma \ref{Modelcond} below. 
\end{proof}

We call \it{left injective}  \rm,  \it{left $\tau$-local}  \rm   and   \it{left $\tau$-$\mathbb{A}^1$-local}   \rm  model structures for the left model structures induced in $s\mbox{\rm coCAlg}_{\cR}(Sm_F)$ in the previous theorem.%

\begin{lemma}\label{Cylinder}
Any morphism $f:F\rightarrow G$ in  $s\mbox{\rm coCAlg}_{\cRR}(Sm_F)$ can be factored as 
\[F\xrightarrow{i} D\xrightarrow{q} G\] where $U(i)$ is a cofibration in $s\Mod_{\cRR}(Sm_F)_{inj}$  and $U(q)$ is a trivial fibration in $s\Mod_{\cRR}(Sm_F)_{proj}$.
\end{lemma}
\begin{proof}
The proof of this result is the mapping cylinder construction this argument goes back to Goerss, Jardine and Raptis in the setting of coalgebras with the section-wise monoidal structure. 

Choose the cylinder object of $F$  in $s\mbox{\rm coCAlg}_{\cRR}(Sm_F)$ given by:
 \[\xymatrix{F\oplus F\ar[r]^-{\nabla} \ar[d]_-{(i_0,i_1)} & F\\
            F\otimes \cRR[\Delta^1] \ar[ru]_p  &  }\] here the coalgebra structure in $\cRR[\Delta^{1}]$ is given by the diagonal map $\Delta^1\rightarrow \Delta^1\times \Delta^1$. The mapping cylinder $Cyl(f)$, is constructed  by the push-out diagram in $s\mbox{\rm coCAlg}_{\cRR}(Sm_F)$. 

\[
\xymatrix{
F\otimes \cRR\ar[r]^-f\ar[d]^-{i_0}  &G\ar[d]^-j\\
F\otimes \cRR[\Delta^1]\ar[r]^-{j}& Cyl(f)}.
\] Since $(f\circ p)\circ i_0=f$ by the universal property of the push-out diagram there exists a unique map  $q: Cyl(f)\rightarrow B$ such that $f\circ p=q\circ j$ and $G\rightarrow Cyl(f)\rightarrow G$ is the identity map.  Then we can construct the factorization of $f:F\rightarrow G$ given by:   
\[F\xrightarrow{ i_1} F\otimes \cRR[\Delta^1]\xrightarrow{j} Cyl(f)\xrightarrow{p} G\] the composition of the first two maps is a section-wise monomorphism in $s\Mod_{\cRR}(Sm_F)$, in other words it is a cofibration in $s\Mod_{\cRR}(Sm_F)_{inj}$.  The map $U(i_0)$ is a trivial cofibration in $s\Mod_{\cRR}(Sm_F)_{inj}$, and trivial cofibrations are preserved under push-outs, then $U(j)$  is a section-wise weak equivalence, then $U(q)$ is also a section-wise weak equivalence, in other words a trivial fibration in $s\mbox{Mod}_\cRR(Sm_F)_{proj}$.

\end{proof}

We want to prove that  the maps which has the right lifting property with respect to all the cofibrations are weak equivalences; this is the remaining condition to check in the proposition. In suffices to show the following proposition.

\begin{proposition}\label{Modelcond}
Let $\cI$ the set of monomorphisms between $\kappa$-presentable objects. Then the $\cI$-inj is contained in the set of weak equivalences $\cW_\cRR$ (resp. $\cW_{\cRR,\tau}$ and $\cW_{\cRR,\tau}^{\mathbb{A}^1}$) 
\end{proposition}
\begin{proof}
Let $f:F\rightarrow G$ be a map in $\cI$-inj, in order to show that $f$ is a weak equivalence in $s\mbox{\rm coCAlg}_{\cRR}(Sm_F)$ it is enough to show that it is a weak equivalence in $s\Mod_\cRR(Sm_F)_{proj}$.  Then it suffices to show that $U(f):U(F)\rightarrow U(G)$ has the right lifting property with respect to the set of generating cofibrations in $s\Mod_\cRR(Sm_F)_{proj}$ given by:
\[ \{\cRR[h_X]\otimes \cRR[\partial\Delta^n]\rightarrow \cRR[h_X]\otimes\cRR[\Delta^n] \,:\,\,\, X\in Ob(Sm_F), \,\,\, \, n\geq 1\}.\] The map $f:F\rightarrow G$ is a $\kappa$-directed colimit of $\kappa$-presentable objects in $\mbox{Arrw}(\mbox{\rm coCAlg}_{\cRR}(Sm_F))$, i.e. $f$ is of the form $\colim_\alpha f_\alpha:\colim_\alpha F_\alpha\rightarrow \colim_\alpha G_\alpha$ where $F_\alpha$ and $G_\alpha$ are $\kappa$-presentable. Since the forgetful functor $U$ preserves colimits then there exists a factorization:

\[\xymatrix{h_X\otimes \cRR[\partial\Delta^n]\ar[r]\ar[d]& U(F_\alpha) \ar[r] \ar[d]& U(F)\ar[d]\\
                h_X\otimes\cRR[\Delta^n]\ar[r] & U(G_\alpha)\ar[r]& U(G) }
 \] By the Lemma \ref{Cylinder}, there exists a factorization in $s\mbox{coCAlg}_\cRR(Sm_F)$
 \[F_\alpha\xrightarrow{i_\alpha} D_\alpha\xrightarrow{p_\alpha} G_\alpha\]  where $U(p_\alpha)$ is a trivial fibration in $s\Mod_\cRR(Sm_F)_{proj}$, then there exists a morphism \[g_{1,X}:\cRR[h_X]\otimes\cRR[\Delta^n]\dashrightarrow U(D_\alpha)\] such that the diagram commutes:
\[\xymatrix{                                          & U(F_\alpha) \ar[r] \ar[d]      & U(F)\ar[dd]\\
               \cRR[h_X]\otimes \cRR[\partial\Delta^n]\ar[r]\ar[d]\ar[ur]& U(D_\alpha)  \ar[d]\ar@{-->}[ur]^-{g_2}& \\
                \cRR[h_X]\otimes\cRR[\Delta^n]\ar[r] \ar@{-->}[ur]^-{g_{1,X}} \ar[r]  & U(G_\alpha)\ar[r]& U(G) }
\]
by the construction of the cylinder object $D_\alpha$ is also $\kappa$-presentable then the map $i_\alpha:F_\alpha\rightarrow D_\alpha$ is an element in $\cI$ then by assumption there exists a morphism $g_2$ in $s\mbox{\rm coCAlg}_{\cRR}(Sm_F)$ which makes the diagram commutes: 
\[\xymatrix{F_\alpha\ar[r]\ar[d]& F\ar[d]\\
                D_\alpha \ar[r]\ar@{-->}[ur]-^{h_2} &G }.
 \]Then the composition $U(g_2)\circ g_1$ provides the lift and then $U(f)$ is a weak equivalence.
\end{proof}

\begin{rmk}
Let $s\Mod_{\cRR}(Sm_F)$ be the category of presheaves of simplicial modules over $Sm_F$, let  $\mathrm{L}s\Mod_{\cRR}(Sm_F)_{inj}$  and $\mathrm{L}s\Mod_{\cRR}(Sm_F)_{proj}$ be the Bousfield localizations of the injective and projective model structure with the same class of weak equivalences $\cW'$. Recall that under the left Bousfield localization the class of cofibrations remains the same, then the class of trivial fibrations is preserved, because the former are the maps which satisfies the right lifting property with respect to all the cofibration. Therefore Lemma \ref{Cylinder} says that every morphism $f:F\rightarrow G$ in $s\mbox{\rm coCAlg}_{\cRR}(Sm_F)$ can be factored as 
\[F\xrightarrow{i} D\xrightarrow{q} G\] where $U(i)$ is a cofibration in $\mathrm{L}s\Mod_{\cRR}(Sm_F)_{inj}$  and $U(q)$ is a trivial fibration in $\mathrm{L}s\Mod_{\cRR}(Sm_F)_{proj}$ and Proposition \ref{Modelcond}  tells that $\cI-inj\subset \cW'$.   Then by Theorem \ref{GenModel} there exists a model category structure in $s\mbox{\rm coCAlg}_{\cRR}(Sm_F)$ where the weak equivalences are  $\cW'$. Furthermore, note that this model structure is the left induced model structure from $\mathrm{L}s\Mod_{\cRR}(Sm_F)_{inj}$. Thus cofibrations in the former model structure remains the same than cofibration in the left induced model structure from $s\Mod_{\cRR}(Sm_F)_{inj}$. Then it is a left Bousfield localization. 

In particular the  \it{left $\tau$-local}  \rm   and   \it{left $\tau$-$\mathbb{A}^1$-local}   \rm  model structures induced in $s\mbox{\rm coCAlg}_{\cR}(Sm_F)$ are left Bousfield localizations of the \it{left injective}  \rm  model structure in $s\mbox{\rm coCAlg}_{\cR}(Sm_F)$.
 \end{rmk}


\subsection{Homology Localization}

\begin{definition}
The $\mathbb{A}^1$-singular chain complex of $\cX$ with $\cRR$ coefficients, denoted by $C_{*}^{\mathbb{A}^1}(\cX,\cRR)$, is defined to be the Nisnevich-$\mathbb{A}^1$-localization $\mathrm{L}_{\mathbb{A}^1}\mathrm{L}_{Nis}(C_{*}(\cX,\cRR))$. We just write  $C_{*}(\cX)$ for $C_{*}(\cX,\mathbb{Z})$. The $\mathbb{A}^1$-homology sheaves of $\cX$ with coefficients in $\cRR$ are defined by $\mathbf{H}_{i}^{\mathbb{A}^1}(\cX,\cRR):=a_{\tau}H_i(C_{*}^{\mathbb{A}^1}(\cX,\cRR))$.   
\end{definition}

\begin{definition}
 A morphism $\cX\rightarrow \cY$ of motivic spaces is called an $\mathbf{H}^{\mathbb{A}^1}\cRR$-homology equivalence if the induced morphism $\mathbf{H}^{\mathbb{A}^1}_{*}(\cX,\cRR)\rightarrow \mathbf{H}^{\mathbb{A}^1}_{*}(\cY,\cRR)$ is an isomorphism. A space $\cZ$ is called $\mathbf{H}^{\mathbb{A}^1}\cRR$-local if fo every $\mathbf{H}^{\mathbb{A}^1}\cRR$-equivalence $\cX\rightarrow \cY$ the induced map
\[Map(\cY,\cZ)\rightarrow Map(\cY,\cZ) \] is a weak equivalence. 
A map $\cX\rightarrow \cX_{\mathbf{H}^{\mathbb{A}^1}\cRR}$ is called $\mathbf{H}^{\mathbb{A}^1}\cRR$-localization if it is an $\mathbf{H}^{\mathbb{A}^1}\cRR$-equivalence and $\cX_{\mathbf{H}^{\mathbb{A}^1}\cRR}$ is $\mathbf{H}^{\mathbb{A}^1}\cRR$-local.
\end{definition}

\begin{rmk}
Let $\mathrm{L}_{\mathbb{A}^1}\mathrm{L}_{\tau}s\Psh(Sm_F)$ be the category of simplicial presheaves endowed with the injective-$\tau$-$\mathbb{A}^1$-local model structure, by \cite{MorelVoevodsky} the category is proper combinatorial model category. Then by the methods in \cite{goerss_jardine_1998} the Bousfield localization with respect to the class of  $\mathbf{H}^{\mathbb{A}^1}\cRR$-homology equivalence should exist.
\end{rmk}


\subsection{$\mathbb{A}^1$-Goerss Theorem for $\cR$ algebraically closed} 

Let $k$ be an algebraically closed field. In this section we will generalize Theorem C in \cite{MR1363853}, \it{i.e}   \rm    the diagonal coalgebra $\cR[\cX]$ determines the homotopy type of $\cX$ up to Bousfield Localization with respect to $\mathbb{A}^1$-homology with coefficients in $\cR$.

\begin{proposition}\label{adjGoerss}
Let $\cRR$ be a commutative ring. Endow the category $s\Psh(Sm_F)$ with the injective (resp. $\tau$-local and $\tau$-$\mathbb{A}^1$-local) model structure and $\scoCAlg$ with the respective model structure from \ref{msCoalgebras}.
Then the adjunction 
\[\cRR^{\delta}[-]:s\Psh(Sm_F)\rightleftarrows \scoCAlg:(-)^{gp}\] is a Quillen adjunction.
\end{proposition}
\begin{proof}

It is enough to show that $\cRR^{\delta}[-]$ preserves cofibrations and acyclic cofibrations. Since the model structure in  $\scoCAlg$ is left induced by the forgetful functor $U:\scoCAlg\rightarrow \sMod(Sm_F)$, where we endow the category  $\sMod(Sm_F)$ with the injective (resp. $\tau$-local and $\tau$-$\mathbb{A}^1$-local) model structure, it is enough to show that $\cRR[-]$ preserves cofibrations and acyclic cofibrations.
This follows from proposition \ref{free-forget}.

\[\xymatrix{s\Psh(Sm_F)\ar[r]^-{\cRR^{\delta}[-]}\ar[rd]_-{\cRR[-]}&\scoCAlg\ar[d]^-U\\
                                              &\sMod(Sm_F)}.\]

\end{proof}

\begin{theorem}\label{GoerssThm}
 Let $k$ be an algebraically closed field. Then the functor below is fully faithful 
\[\cR^{\delta}[-]:\mathrm{L}_{\mathbf{H}^{\mathbb{A}^1}\cR}\mathrm{L}_{mot}s\Psh(Sm_F)\rightarrow \mathrm{L}_{mot}\scoCAlg.\]  Furthermore, for every motivic space $\cX$ the derived unit map 
\[X\rightarrow (\cR^{\delta}[\cX]^{fib})^{gp} \] exhibits the target as the $\mathbf{H}^{\mathbb{A}^1}\cR$-localization of $\cX$. 
\end{theorem}
\begin{proof}
First let us show that $\cR^\delta[-]:\mathrm{L}_{\mathbf{H}^{\mathbb{A}^1}\cR}\mathrm{L}_{\mathrm{mot}}s\Psh(Sm_F)\rightarrow \scoCAlg$ is a left Quillen functor. $\cR^{\delta}[-]$ preserves cofibrations because the cofibrations in $\mathrm{L}_{\mathbf{H}^{\mathbb{A}^1}\cR}\mathrm{L}_{\mathrm{mot}}s\Psh(Sm_F)$ are the same than cofibrations in $\mathrm{L}_{\mathrm{mot}}s\Psh(Sm_F)$. It also preservers trivial cofibrations by the definition of $\mathbf{H}^{\mathbb{A}^1}\cR$-weak equivalences (resp. $\mathbf{H}\cR$, $\mathbf{H}_{\tau}\cR$). Thus $(-)^{gp}$ is a right Quillen functor in particular preserves fibrations and trivial fibrations.

Let us show now that given $\cX$ a motivic space (respectively $inj-fib$, $\tau$-fib), the derived unit $\cX\rightarrow(\cR^{\delta}[\cX]^{fib})^{gp}$ exhibits the target as a the $\mathbf{H}\cR$-Bousfield Localization. 

\it{Claim:}  \rm $(-)^{gp}$ sends injective (resp. $\tau$, $\mathbb{A}^{1}$-$\tau$) weak equivalences to $\mathbf{H}^{\mathbb{A}^1}_{\tau}\cR$ (resp. $\mathbf{H}\cR$, $\mathbf{H}_{\tau}\cR$) weak equivalences. 

Granting this, it follows immediately  that given $\cX\in s\Psh(Sm_F)$ 
\[\xymatrix{\cX=(\cR^{\delta}[\cX])^{gp}\rightarrow (\cR[\cX]^{fib})^{gp}&}  \mbox{(resp. $\tau$-fib, $\tau$-$\mathbb{A}^1$-fib)}\]   is an $\mathbf{H}\cR$- injective (resp. $\tau$, $\mathbb{A}^{1}$-$\tau$) weak equivalences in $s\Psh(Sm_F)$.

It remains to show that $(\cR^{\delta}[\cX]^{\mathbb{A}^1\mbox{-\footnotesize{fib}}})^{gp}$ is an $\mathbf{H}\cR$-local space. We factor the counit map $\epsilon:\cR[\cX]\rightarrow \cR$ by a trivial-cofibration followed by fibration $\cR[\cX]\rightarrow \cR[\cX]^{fib}\rightarrow \cR$. Applying the group-like functor $(-)^{gp}$ we get

\[\xymatrix{\cX \ar[r]&(\cR[\cX]^{fib})^{gp}\ar[r]& {*}&} \mbox{(resp. $\tau$-fib, $\tau$-$\mathbb{A}^1$-fib)}.\]
We already show that the first map is an $\mathbf{H}\cR$-injective (resp. $\mathbf{H}_{\tau}\cR$, $\mathbf{H}^{\mathbb{A}^1}_{\tau}\cR$ ) weak equivalence. Furthermore, since  $(-)^{gp}$ is a right Quillen functor, it preserves fibrations. Then $(\cR[\cX]^{fib})^{gp}$ is $\mathbf{H}\cR$ (resp. $\mathbf{H}_{\tau}\cR$, $\mathbf{H}^{\mathbb{A}^1}_{\tau}\cR$) local space.

\it{Proof of the claim:}  \rm Let $\alpha: C\rightarrow D$ be an injective (resp. $\tau$, $\mathbb{A}^{1}$-$\tau$) weak equivalence in $\scoCAlg$. By definition it suffices to show that $\cR[C^{gp}]\rightarrow \cR[D^{gp}]$ is a an injective (resp. $\tau$, $\mathbb{A}^{1}$-$\tau$) weak equivalence.    Since  every element in $\scoCAlg$ is an objectwise  coalgebra, by the Theorem \ref{etalesplit} $\acute{E}t$ induces a functor in $\scoCAlg$, which has a natural splitting. This implies that $\acute{E}t(\alpha)$ is a retract of $\alpha$. Then $\acute{E}t(\alpha)$ is an injective (resp. $\tau$, $\mathbb{A}^{1}$-$\tau$) weak equivalence in $\scoCAlg$.
\[\xymatrix{\acute{E}t(C)\ar[r]\ar[d]_{\acute{E}t(\alpha)}&C\ar[d]^{\alpha}\\
                 \acute{E}t(D)\ar[r]&D}\] Since we are assuming $\cR$ to be an algebraically closed field by Proposition \ref{ShvEtgplike} 
 \[\xymatrix{\cR^{\delta}[(C)^{gp}]\cong \acute{E}t(C)\ar[r]^-{\acute{E}t(\alpha)}&\acute{E}t(D)\cong\cR^{\delta}[(D)^{gp}]}.\] Then   $\cR[C^{gp}]\rightarrow \cR[D^{gp}]$ is an injective (resp. $\tau$, $\mathbb{A}^{1}$-$\tau$) weak equivalence.

\end{proof}

\begin{rmk}

Note that it was enough to show that $\cR[C^{gp}]\rightarrow \cR[(C^{fib})^{gp}]$ is an injective (resp. $\tau$, $\mathbb{A}^{1}$-$\tau$) weak equivalence. In general it is hard to give an explicit fibrant replacement for coalgebras. At least up to our knowledge, we could not construct an explicit one. But since we have a good knowledge of the category of coalgebras over an algebraically closed field $k$, $\coCAlg_\cR$, Proposition \ref{Etgplike} allows us to show that $(-)^{gp}$ sends  injective (resp. $\tau$, $\mathbb{A}^{1}$-$\tau$) weak equivalences between coalgebras to $\mathbf{H}^{\mathbb{A}^1}_{\tau}\cR$ (resp. $\mathbf{H}\cR$, $\mathbf{H}_{\tau}\cR$) weak equivalences.

\end{rmk}

\section{Discrete $G$-objects and $\mathbb{A}^1$-Goerss theorem for $k$ non-algebraically closed}\label{gmot}
Let $G$ be a profinite group. In this section we will define the notion of discrete $G$-objects for the categories of abelian groups, sheaves of sets, sheaves of abelian groups, simplicial sets $\sS$ and simplicial sheaves. The aim is to understand the notion of discret $G$-motivic spaces. 

\subsection{Discrete $G$-Spaces}

Let us take the category of simplicial discrete $G$-$\Sets$

\begin{equation}
sG\mbox{-}\Sets_{d}\simeq s\Sh(G\mbox{-}\Sets_{fd})
\end{equation}

\begin{definition}\label{G-dismod}
Let $f:F\rightarrow G$ be a morphism between simplicial discrete $G$-sets:
\begin{itemize}
\item $f$ is a weak equivalence if and only if it is a weak equivalence between the underlying simplicial sets.
\item  $f$ is a cofibration if and only if it is a injection.
\item $f$ is a fibration if and only if it has the right lifting property with respect to all trivial cofibrations.
\end{itemize} 
\end{definition}

\begin{proposition}
The class of weak equivalence, fibrations and cofibrations from Definition \ref{G-dismod}  defines a simplicial model model structure in $sG\mbox{-}\Sets_{d}$.
\end{proposition}
\begin{proof}
This is proved in \cite{MR1320993}. Since the category of simplicial discrete $G$-sets is the category of simplicial sheaves $s\Sh(G\mbox{-}\Sets_{fd})$, an alternative proof follows from \cite{MR906403}. Recall that there is only one stalk, the forgetful functor.
\end{proof}

\begin{definition}
Let $sG\mbox{-}\Sets_{d}$ be the category of simplicial discrete $G$-spaces endowed with the Jardine model structure. Consider its homotopy category $\mathrm{Ho}((sG\mbox{-}\Sets_{d})$ an element in the homotopy category is called a discrete $G$-space.
\end{definition}

\subsection{Discrete $G$-Motivic spaces}

Let us consider the category of simplicial presheaves on the product site $G\mbox{-}\Sets_{fd}\times Sm_F$, where $Sm_F$ is endowed with a Grothendieck topology $\tau$ and $G\mbox{-}\Sets_{fd}$ is endowed with the cofinite topology, by abuse of notation we will denote the cofinite topology for $G$ a profinte group as $G$

By \cite{MR906403}, the category $s\Psh( Sm_F\times G\mbox{-}\Sets_{fd})$ is endowed with injective model structure. Note that, if we consider the trivial Grothendieck topology in $Sm_F$, the fibrant objects in the $triv\times G$-$inj$ model structure are elements in $s\Psh(Sm_F, sG\mbox{-}\Sets_{d})$.

Furthermore, as in the definition of discrete $G$-spaces, we can define a discrete $G$-simplicial sheaf over the site $(Sm_F,\tau)$ as an element in the homotopy category:

As usual we  introduce  the $\mathbb{A}^1$-localization in the homotopy category

\[\mathrm{Ho}((\mbox{L}_{\tau\times G}s\Psh(Sm_F\times G\mbox{-}\Sets_{fd}))\]

\begin{definition}
A presheaf of simplicial discrete $G$-sets $\cX\in \Psh(Sm_F,sG\mbox{-}\Sets_{fd})$ is $\mathbb{A}^1$-local if for any $X\in Sm_F$ the induced 
\[\mapGd(h_U,\cX)\rightarrow \mapGd(h_{U\times \mathbb{A}^1},\cX)\]
is a weak equivalence, where $h_U$ and $h_{U\times\mathbb{A}^1}$ are consider  as presheaves with the trivial action by $G$. 
\end{definition}

\begin{definition}
Discrete $G$-motivic spaces are fibrant objects in the model category 
\[\mbox{L}_{\mathbb{A}^1}L_{\tau\times G}\Psh(Sm_F,sG\mbox{-}\Sets_{fd})\].
\end{definition}

\begin{definition}
Let $(\cC,\tau)$ be an small Grothendieck site. A discrete $G$-sheaf  $F$ over $(C,\tau)$ is a sheaf over the site $F\acute{E}t/\cR\times(\cC,\tau)$.
\end{definition}

\begin{definition}
Let $G=Gal(k_{sep}/k)$ be the absolute Galois group of a field $k$ and $F$  a perfect field, consider $(Sm_F,\tau)$ with $\tau$ a Grothendieck topology over $Sm_F$. We define the category of $G$-discrete simplicial $\tau$-local presheaves as the category $s\Psh(F\acute{E}t/\cR\times (Sm_F,\tau))$ endowed with the  $\acute{e}t\times\tau$-injective model structure. We denote this category as $s\Psh(Sm_F)_{\tau\times G-inj}$.
\end{definition}

\begin{definition}
Let $G=Gal(k_{sep}/k)$ be the absolute Galois group of a field $k$ and $F$  a perfect field, consider $(Sm_F,\tau)$ with $\tau$ a Grothendieck topology over $Sm_F$. Let $S$ be a collection of morphisms in $\mathrm{Ho}((s\Psh(Sm_F)_{\tau-inj})$. We will say that a morphism $f:\cX\rightarrow \cY\in s\Psh(Sm_F)_{\tau\times G-inj}$ is an $S$-local equivalence if the morphism at the stalk $\cX_{\Spec(\cR_{sep})}\rightarrow \cY_{Spec({\cR}_{sep})}$ is an $S$-local weak equivalence in $s\Psh(Sm_F)_{\tau-inj}$. 
\end{definition}

\begin{example}
We are particular interested in the following two examples.
\begin{enumerate}
\item Let $S:=\{X\times \mathbb{A}^1\rightarrow X\,\,:\,\, X\in Sm_F\}$, then a morphism $f:\cX\rightarrow\cY$ between $G$-discrete spaces is $\mathbb{A}^1$-equivalence if $F_{\Spec(\bar{\cR})}:\cX_{\Spec(\bar{\cR})}\rightarrow \cY_{\Spec(\bar{\cR})}$ is an $\mathbb{A}^1$-equivalence.

 \item Let $E$ be a motivic homology theory, a morphism $\cX\rightarrow \cY$ between $G$-discrete spaces is be an $E$-local weak equivalence if the morphism on the stalk $F_{\Spec(\bar{\cR})}\rightarrow G_{\Spec(\bar{\cR})}$ is a $E$-local weak equivalence, \it{i.e} \rm the induced morphism $E_{*}(\cX_{\Spec(\bar{\cR})})\to E_{*}(\cY_{\Spec(\bar{\cR})})$ is an isomorphism.  

 \end{enumerate}
\end{example}

\begin{rmk}
The $S$-local fibrant replacement in the category  $s\Psh(Sm_F)_{\tau-inj\times G}$ is more complicated. Already in the category of discrete spaces $G-\cS_d$ the Eilenberg-Mac Lane space $K(\mathbb{F}_p,n)$ is not fibrant object in $\sSGd$ for $G=Gal(\bar{\mathbb{F}}_p/\mathbb{F}_p)$, although is a fibrant object in $L_{H\mathbb{F}_p}\cS$ (see \cite[Example 7.3]{MR1320993}). This is a consequence of the fact that  $Gal(\bar{\mathbb{F}}_p/\mathbb{F}_p)\simeq \hat{\mathbb{Z}}$ is of cohomological dimension 1. 
\end{rmk}

\subsection{Homotopy Fixed points for discrete $G$-Motivic Spaces}
Here we define and discuss elementary properties of homotopy fixed points for discrete $G$-motivic spaces, this depends in the properties of discrete $G$-spaces. 

\rm Let $G=Gal(k_{sep}/k)$ be the absolute Galois group of a field $k$ and $F$  a perfect field, consider $(Sm_F,\tau)$ with $\tau$ a Grothendieck topology over $Sm_F$. We have a canonical functor given by the constant sheaf, in other words we can endowed every simplicial presheaf over $Sm_F$ with the trivial action.  
\begin{equation}\label{fixed}
\xymatrix{s\Psh(Sm_F)_{\tau-inj}\ar@<0.5ex>[rr]^-{constant}&&s\Psh( Sm_F\times G\mbox{-}\Sets_{fd})_{G\times\tau-inj}
.\ar@<0.5ex>[ll]^-{(-)^G}}
\end{equation}

The right adjoint is given by the sections $\cX(\Spec(\cR))$ in other words; it is given by the fixed points $\cX^G$. 
Since the cofibrations in  $s\Psh( Sm_F\times G\mbox{-}\Sets_{fd})_{G\times\tau-inj}$ are section-wise and the constant sheaf functor preserves weak equivalences, the adjunction \ref{fixed} induces a Quillen adjunction. We have a well-defined adjoint pair on the homotopy categories
\begin{equation}
\xymatrix{\mathrm{Ho}(s\Psh(Sm_F)_{\tau-inj})\ar@<0.5ex>[rr]^-{constant}&&\mathrm{Ho}(s\Psh( Sm_F\times G\mbox{-}\Sets_{fd}))
.\ar@<0.5ex>[ll]^-{(-)^G}}
\end{equation}

\begin{definition}
Let $\cX\in \mathrm{Ho}(s\Psh( Sm_F\times G\mbox{-}\Sets_{fd}))$ be a $G$-discrete object. Define the homotopy fixed objects as:
\[\cX^{hG}:=(\cX^{fib})^{G}\] where the fibrant replacement is taken in the category $s\Psh( Sm_F\times G\mbox{-}\Sets_{fd})_{\tau\times G\mbox{-}inj}$.
\end{definition}

\subsection{$\mathbb{A}^1$-Goerss Theorem for $\cR$ non-algebraically closed}

\begin{proposition}\label{adjG}
There is a left adjoint functor $\bar{\cR}^\vee[-]_G:\Psh(Sm_F,\sSGd)\rightarrow s\coCAlg_\cR(Sm_F)$ which sends $G/H$ to the constant presheaf of coalgebras  $(\bar{k}^H)^{\vee}$ placed in simplicial degree zero.      

\end{proposition}
\it{Notation:}   \rm  Here the dual $\cR$-vector spaces  $(L)^{\vee}$ is regarded as coalgebra with the comultiplication induced by the dual of the multiplication map.

\begin{proof}
Let $\cX\in\Psh(Sm_F,\sSGd)$ for every $U\in Sm_F$, from \ref{final}  we define the presheaf of coalgebras  $\bar{\cR}^\vee[\cX]_G$ section-wise.  
\[\bar{\cR}^\vee[\cX]_G(U):=\bar{\cR}^\vee[\cX(U)]_G.\] 
\end{proof}

\begin{proposition}\label{GadjGoerss}
Let $k$ be a perfect field. Endow the category $s\Psh(Sm_F,\sSGd)$ with the injective (resp.$ G\times\tau$-local and $G\times\tau$-$\mathbb{A}^1$-local) model structure and $s\coCAlg_\cR$ with the  injective (resp.$\tau$-local and $\tau$-$\mathbb{A}^1$-local) from \ref{msCoalgebras}.
Then the adjunction 
\[\bar{\cR}^\vee[-]_G:s\Psh(Sm_F,\sSGd)\rightleftarrows s\coCAlg_k:R\] is a Quillen adjunction.
\end{proposition}

\begin{proof}

It is enough to show that $\bar{\cR}^\vee[-]_G$ preserves cofibrations and acyclic cofibrations. Since the model structure in  $\scoCAlg$ is left induced by the forgetful functor $U:\scoCAlg\rightarrow \sMod(Sm_F)$, where we endow the category  $\sMod(Sm_F)$ with the injective (resp. $\tau$-local and $\tau$-$\mathbb{A}^1$-local) model structure, it is enough to show that $\bar{\cR}^\vee[-]_G\circ U$ preserves cofibrations and acyclic cofibrations. Note that  for $X$ a discrete $G$-Set the underlying vector space $\bar{\cR}^\vee[X]_G\circ U$ is isomorphic to $k[X]$. Then the proposition follows from proposition \ref{free-forget}.


\end{proof}


\begin{theorem}\label{G-Goerss}

The functor $\bar{\cR}^\vee[-]_G:\mathrm{L}_{\mathbb{A}^1}\mathrm{L}_{\tau}s\Psh(Sm_F,\sSGd)\rightarrow s\coCAlg_k$ sends $\mathbf{H}_{\tau}^{\mathbb{A}^1}\cR$-equivalences to $\mathbb{A}^1$-$Nis$-weak equivalences of coalgebras and thus induces a functor:
\[\mathrm{L}\bar\cR^{\vee}[-]_G:\mbox{\rm L}_{\mathbf{H}^{\mathbb{A}^1}\cR}\mathrm{Ho}((Spc^{G}_{\bullet}(F))\rightarrow \mathrm{Ho}((\mathrm{L}_{\mathbb{A}^1}\mathrm{L}_{\tau}s\coCAlg_\cR(\mathrm{Sm}_F)).\]
This functor is fully faithful.

Furthermore, for every motivic space $\cX$ the derived unit map 
\[\cX\rightarrow R((\bar\cR^{\vee}[\cX]_G)^{fib}) \] exhibits the target as the $\mathbf{H}^{\mathbb{A}^1}\cR$ localization of $\cX$ in discrete $G$-motivic spaces.

\end{theorem}

\begin{proof}
First let us show that $\bar{\cR}^\vee[-]_G:\mathrm{L}_{\mathbb{A}^1}\mathrm{L}_{\tau}s\Psh(Sm_F,\sSGd)\rightarrow s\coCAlg_k$ is a left Quillen functor. From Proposition \ref{adjG} we notice that the functor is given section-wise; $\bar{\cR}^\vee[-]_G$ preserves cofibrations because the cofibrations in $\mathrm{L}_{\mathbf{H}^{\mathbb{A}^1}\cR}\mathrm{L}_{\mathbb{A}^1}\mathrm{L}_{\tau}s\Psh(Sm_F,\sSGd)$ are the same than cofibrations in $\mathrm{L}_{\mathbb{A}^1}\mathrm{L}_{\tau}s\Psh(Sm_F,\sSGd)$. It also preserves trivial cofibrations, since by the proof in \ref{GadjGoerss} we know that $\bar{\cR}^\vee[\cX]_G\simeq k[\cX]$ as presheaves of vector spaces.Then it follows from the definition of $\mathbf{H}^{\mathbb{A}^1}\cR$-weak equivalences (resp. $\mathbf{H}\cR$, $\mathbf{H}_{\tau}\cR$). 
Thus $R$ is a right Quillen functor in particular preserves fibrations and trivial fibrations.
To show that the functor \[\mathrm{L}\bar\cR^{\vee}[-]_G:\mbox{\rm L}_{\mathbf{H}^{\mathbb{A}^1}\cR}\mathrm{Ho}((Spc^{G}_{\bullet}(F))\rightarrow \mathrm{Ho}((\mathrm{L}_{\mathbb{A}^1}\mathrm{L}_{\tau}s\coCAlg_\cR(\mathrm{Sm}_F))\] is fully faithful Quillen functor, we have to prove that  

\begin{equation}
 \cX\mapsto R((\bar\cR^{\vee}[\cX]_G))\mapsto R((\bar\cR^{\vee}[\cX]_G)^{fib})
 \end{equation} is a weak equivalence.

By Proposition \ref{final} the first morphism is an isomorphism, it remains to show that $R$ sends injective (resp. $\tau$, $\mathbb{A}^{1}$-$\tau$) weak equivalences to $\mathbf{H}^{\mathbb{A}^1}_{\tau}\cR$ (resp. $\mathbf{H}\cR$, $\mathbf{H}_{\tau}\cR$) weak equivalences. Again by \ref{final} the counit of the adjunction is given by $\acute{E}t(C)\rightarrow C$. Then the proof follows as in \ref{GoerssThm}

\end{proof}

\appendix
\section{Recollections on combinatorial model categories}
\subsection{Compactness, Presentability and Accessible Categories}

\begin{definition}
Let $\kappa$ be a regular cardinal and  $\cJ$  a $\kappa$-filtered partially order set.  Let $\cC$ be a category which admits small colimits and let $X$ be an object of $\cC$. Let $\{Y_{\alpha\in\cJ}\}$   be a diagram in $\cC$ indexed by $\cJ$.  Let $Y=\colim_{\alpha\in\cJ}Y_{\alpha}$ be  the colimit of this diagram. There is an  associated map of sets
\[\phi:\colim_{\alpha}\hom_\cC(X, Y_\alpha)\rightarrow \hom_\cC(X,Y)\]
We say that $X$ is $\kappa$-\it{compact} \rm if $\phi$ is a bijective map of sets for every $\kappa$-filtered partial order set $\cJ$ and every diagram $\{Y_{\alpha}\}$ indexed by $\cJ$. We say that $X$ is \it{small} \rm if it is $\kappa$-compact for some small regular cardinal $\kappa$. 
\end{definition}

\begin{definition}
A category $\cC$ is locally presentable if it satisfies the following conditions:
\begin{enumerate}
\item The category $\cC$ admits all small colimits.
\item There exists a small set $S$ of objects of $\cC$ which generates $\cC$ under colimits in other words every object of $C$ may be obtained as the colimit of small diagrams taking values in $S$. 
\item Every object in $\cC$ is small. This is equivalent to say that every object in $S$ is small.
\item For any pair of objects $X,Y\in \cC$ the set $\hom_{\cC}(X,Y)$ is small.  
\end{enumerate}
\end{definition}

\begin{rmk}\label{slimits}
A no-trivial consequence of presentability for a category $\cC$ is that it also admits all small limits.
\end{rmk}

\subsection{Model categories}
\begin{definition}
A \it model category \rm is a category $\cC$ which is equipped with three distinguished classes of morphisms in $\cC$, called \it cofibrations\rm,  \it fibrations \rm and \it weak equivalences, in which the following axioms are satisfied:
\begin{enumerate}
\item The category $\cC$ admits small limits and colimits.
\item Given a composable pair of maps $X\rightarrow Y\rightarrow Z$, if any two of $g\circ f$, $f$ and $g$ are weak equivalences, then so is the third.
\item Suppose $f\rightarrow Y$ is a retract of $g:X'\rightarrow Y' $, that is suppose that there exists a commutative diagram:
\[\xymatrix{ X\ar[r]^i \ar[d]^f \ar@/^1.1pc/[rr]  & X'\ar[r]^i \ar[d]^g& X \ar[d]^f\\
Y\ar[r]^{i^\prime} \ar@/_1.1pc/[rr]   & Y^\prime \ar[r]^{r^\prime} &Y 
}
\]

\item Given a solid  diagram 
\[\xymatrix{ A\ar[r] \ar[d]^i &X\ar[d]^p\\
B\ar[r]\ar@{-->}[ur] & Y
}
\] a dotted arrow can be found making the diagram commute if either:
\begin{enumerate}
\item The map $i$ is a cofibration and the map $p$ is both a fibration and a weak equivalence.
\item The map $i$ is both a cofibration and a weak equivalence, and the map $p$ is a fibration.
\end{enumerate} 
\item Any map $X\rightarrow Z$ in $\cC$ admits factorizations 

\[X\rightarrow Y\rightarrow Z\]
\[X\rightarrow Y\rightarrow Z\]
where $f$  is a cofibration, $g$ is a fibration and a weak equivalence, $f'$  is a cofibration and a weak equivalence and $g'$ is a fibration.
\end{enumerate} 
\end{definition}

\begin{definition}

The homotopy category is defined as follows:

\begin{enumerate}
\item The objects of $h\cC$ are the fibrant-cofibrant objects of $\cC$. 
\item For $X,Y\in h\cC$, the set $\hom_{h\cC}(X,Y)$ is the set of homotopy equivalences classes on $\hom(X,Y)$
 
\end{enumerate}
\end{definition}
\subsection{Properness and Homotopy Push out squares}

\begin{definition}
A model category $\cC$ is left proper if weak equivalences are stable under push-out along cofibrations and it is called right proper if weak equivalences are stable under pullbacks along fibrations.
\end{definition}

\subsection{Quillen adjunctions and Quillen equivalences}
Let  $\cM$ and $\cN$ be two model categories and suppose we are given a pair of adjoint functors

\[L:\cM\rightleftarrows \cN:R\] with $L$ a left adjoint and $R$ a right adjoint. The following conditions are equivalent:

\begin{itemize}
\item The functor $L$ preserves cofibrations and trivial cofibrations 
\item The functor $R$ preserves fibrations and trivial fibrations
\item The functor $L$ preserves cofibrations and the functor $R$ preserves fibrations 
\item The functor $L$ preserves trivial fibrations and the functor $R$ preserves trivial fibrations. 
\end{itemize}
 
\begin{definition}
Let  $\cM$ and $\cN$ be two model categories and suppose we are given a pair of adjoint functors

\[L:\cM\rightleftarrows \cN:R.\] If any of the equivalent conditions above is satisfied. Then we say that the pair $(F,G)$ is a \it{Quillen adjunction}  \rm  between $\cM$ and $\cN$. 
\end{definition}

\subsection{Combinatorial model categories}

\begin{definition} 
A model category $\mcM$ is called combinatorial if it is locally presentable and it is cofibrantly generated \it i.e.\rm:
\begin{itemize}
\item There exists a set $\cI$ of generating cofibrations, \it i.e. \rm the collection of all cofibrations in $\mcM$  is the smallest weakly saturated class of morphisms containing $\cI$ 
\item There exists a set $\cJ$ of generating trivial cofibrations, \it i.e. \rm the collection of all trivial cofibrations in $\cM$ is the smallest weakly saturated class  of morphisms containing $\cJ$. 
\end{itemize}
\end{definition}

A combinatorial model structure is uniquely determined by the generating cofibrations and generating trivial cofibrations.  In \cite[\S A.2.6]{lurie2009higher}  the definition is reformulated in order to emphasize in the  class of weak equivalences which are easier to describe. More concretely they prove the following proposition.

Recall that given a presentable category  and $\kappa$ a regular cardinal, it is said that full subcategory  $\cC_0\subset\cC$  is  $\kappa$-\it{accessible subcategory} \rm of $\cC$ if satisfies the following conditions:
\begin{enumerate}
\item $\cC_0$ is stable under $\kappa$-filtered colimits.
\item There exists a small subset of objects of $C_0$ which generates $C_0$ under $\kappa$-filtered colimits.
\end{enumerate}

If the subcategory $\cC_0$ satisfies (1), this second condition is equivalent to say that:

\begin{itemize}[label=$(2_\tau)$]
\item Let $A$ be a $\tau$-filtered partially ordered set and $\{X_\alpha\}_{\alpha\in A}$ a diagram of $\tau$-compact objects of $\cC$ indexed by $A$. For every $\kappa$-filtered subset $B\subset A$ we let $X_B$ denote the ($\kappa$-filtered) colimit of the diagram $\{X_\alpha\}_{\alpha\in B}$. Furthermore suppose that $X_A$ belongs to $\cC_0$. Then for every $\tau$-small subset $C\subset A$, there exists a $\tau$-small $\kappa$-filtered subset $C\subset B\subset A$ , such that $X_B$ belongs to $C_0$.  
\end{itemize}

\begin{rmk}
This characterization  gives the immediate consequence that for every $\kappa$-filtered colimit preserving functor $f:\cC\rightarrow \cD$  between presentable categories and let $\cD_0\subset\cD$ be a $\kappa$-accessible subcategory. Then $f^{-1}(\cD_0)$ is a $\kappa$-accessible subcategory.  
\end{rmk}

\begin{proposition}\label{smith}[Bousfield, Smith, Lurie]
Let $\cM$ be a presentable category endowed with a model structure. Assume that there exists an small set which generates the collection of cofibrations in $\cC$. Then the following are equivalent:
\begin{enumerate}
\item The model category $\cM$ is combinatorial.
\item The collection of weak equivalences in $\cM$ determines an accessible subcategory if $\cM^{[1]}$ (the category of morphisms on $\cM$).

\end{enumerate}
\end{proposition}
\begin{proof}
This is proved in \cite[Corollary A.2.6.9]{lurie2009higher}.
\end{proof}

Observe that the subcategory of weak equivalences $W\subset \cM^{[1]}$ in a combinatorial model category $\cM$ is an accessible category. Then we have the following proposition:

\begin{theorem}{(Bousfield, Smith, Lurie)}
Let $S$ be a class of morphisms in a combinatorial model category $\cC$ with corresponding full subcategory $\cC^{0}\subset \cC$ of $S$-local objects. Then the following conditions are equivalent:
\begin{enumerate}
\item $\cC^0\subset \cC$ is a colocalization and $C^{0}$ presentable. 
\item $\cC^0\subset \cC$ is a colocalization and the inclusion preserves $\kappa$-filtered colimits for some regular cardinal $\kappa$.
\item There exists a small set $S^0\subset S$ such that an object in $\cC$ is $S$-local precisely if it is $S^0$-local, equivalently $\bar{S^0}\subset \bar{S}$.
\item There exists a colimit preserving functor  $F:\cC\rightarrow \cD$ to a combinatorial model category $\cD$ such that $\bar{S}$ consist of those morphisms which are sent to equivalences by $F$. 
\end{enumerate}

\end{theorem}

\begin{rmk}
If we can construct a functor $F:\cC\rightarrow \cD$ which satisfies the conditions of the situation $(4)$ then we can guarantee the existence of the Bousfield localization, but we can not characterize the local objects in terms of such functor. 
\end{rmk}

For doing that lets recall the following proposition from \cite{lurie2009higher}:
\begin{theorem}\label{GenModel}
Let $\mcM$ a locally presentable category and let $W$ and $C$ be classes of morphisms in $\mcM$ with the following properties:
\begin{enumerate}
\item  The collection $C$ is a weakly saturated class of morphisms of $\mcM$ of morphisms of $\mcM$, and there exists a small subset $C_0\subset C$ which generates $C$ as a weakly saturated class of morphisms.
   
\item The intersection $C\cap W$ is a weakly saturated class of morphisms of $\mcM$ 
\item The full subcategory $W\in \mcM^{[1]}$ is an accessible subcategory of $\mcM^{[1]}$.
\item The class $W$ has the two-out-of-three property. 
\item If $f$ is a morphism in $\mcM$ which has the right lifting property with respect to each element of $C$, then $f\in W$. 
\end{enumerate}
Then $\mcM$ admits a combinatorial model structure, where the weak equivalences are the elements of $C$ and the weak equivalences in $\mcM$ are the elements of $W$ and a morphism is a fibration if and only if it has the right lifting property with respect to every morphisms in $C\cap W$.
\end{theorem}
\begin{proof}
\cite[Proposition A.2.6.8 ]{lurie2009higher}
\end{proof}
 
\begin{rmk}
The theorem \ref{GenModel} is useful to create new model structures on a locally presentable categories. As is pointed out in \cite{MR3073905} this theorem does not assume have a given  explicit set of generating trivial cofibrations but the existence depends on the accessibility of the class of weak equivalences. Usually, the condition of the accessibility of the weak equivalences is not that easy to verify but for our purposes the following proposition will be useful.   \end{rmk}

\begin{proposition}\label{Modelinverse}
Let $F:\cC\rightarrow \cD$ be an accessible functor and $D'$ an accessible  and accessible embedding subcategory of $\cD$. Then $F^{-1}(\cD')$ is an accessible and accessible embedding subcategory of $\cC$.  
\end{proposition}
\begin{proof}
\cite[Remark 2.50]{lurie2009higher}
\end{proof}
\begin{rmk}
In particular if $\cD$ is a combinatorial model category and let $\cD^{[1]}$ be the category of morphisms in $\cD$, then by \cite[Corollary A.2.6.6]{lurie2009higher}  the full subcategory spanned  by the weak equivalences   $W\in A^{[1]}$, the full subcategory spanned  by the fibrations $F\in A^{[1]}$ and $F\cup W$ are accessible subcategories of $\cD^{[1]}$. By  \cite[Proposition 2.23]{jiris1994} each left or right adjoint between accessible categories is an accessible functor, this will be useful to induce model structures by left or right adjoint. 
\end{rmk}

\begin{definition}
Let $\mcM$ and $\mcM_{loc}$ be two model structures in the same underlying category. We say that $\mcM_{loc}$ is a \it left Bousfield localization \rm of $\mcM$, if the following conditions are satisfied:
\begin{itemize}
\item A morphism $f$ is a cofibration  in $\mcM$ if and only if $f$ is a cofibration in $\mcM_{loc}$.
\item If  a morphism $f$ is a weak equivalence on  $\mcM$, then $f$ is a weak equivalence in $\mcM_{loc}$. 
\end{itemize}
\end{definition}

\begin{proposition}\label{comblocal}
Let $\mcM$ be a left proper combinatorial simplicial model category. Then, every combinatorial Bousfield localization of $\mcM$ has the form $S^{-1}\mcM$, where $S$ is some small set of cofibrations in $\mcM$.
\end{proposition}
\begin{proof}
This proof is giving in \cite[Proposition A.3.7.4]{lurie2009higher}. 
\end{proof}


\subsection{Diagram Categories and Homotopy Limits and Colimits }

\begin{definition}
Let $\cI$ be an small category and  $\cM$ a combinatorial model category.  We will that a natural transformation   $\alpha:F\rightarrow G$ is:

\begin{itemize}
\item a \it level-wise weak equivalence \rm if $ f(C):\cX(C)\rightarrow\cY(C)$ is a weak equivalence in $\cM$ for each $C\in \cI$. 
\item an \it injective  cofibration \rm if  $\alpha(C):F(C)\rightarrow G(C)$  is a cofibration in $\cM$ for each $C\in \cI$. 
\item a \it projective fibration  \rm if  $\alpha(C):F(C)\rightarrow G(C)$  is a fibration in $\cM$ for each $C\in \cI$. 
\item an \it injective fibration \rm if it has the right lifting property with respect to every morphism $\alpha\in \mbox{Fun}(\cI,\cM)$ which that is both a level weak equivalence and an injective  cofibration. 
\item a \it projective cofibration \rm it it has the left lifting property with respect to every morphism $\alpha\in\mbox{Fun}(\cI,\cM)$  which is both a  level weak equivalence and a projective fibration. 
\end{itemize}
\end{definition}

\begin{proposition}\label{proinj}
Let $\cI$ be a small category and $\cM$ a combinatorial model category. There exist two combinatorial model structures on $\mbox{Fun}(\cI,\cM)$. 
 \begin{itemize}
 \item The  \it projective model structure determined by the level-wise weak equivalences, projective fibrations and projective cofibrations. 
 
  \item  The \it injective model structure  determined by the level-wise weak equivalence, injective cofibrations, injective fibrations.
   \end{itemize}
\end{proposition}
\begin{proof}
This is proved in \cite[Proposition A.2.8.2]{lurie2009higher}.
\end{proof}

\begin{rmk} 
\begin{itemize}

\item\rm Let $\cI$ be an essentially small category and $\cM$ a combinatorial model category. Then the proposition above is valid for  the functor category  $\mbox{Fun}(I,M)$ \cite[Proposition A.2.8.2]{lurie2009higher}. These model structures are useful to define \it homotopy limits \rm  and \it homotopy colimits. \rm

\rm

\item It follows from the definitions that the class of projective cofibrations is contained in the class of injective cofibrations and dually the class of injective fibrations is contained in the class of projective fibrations. Then it induces a Quillen adjunction given by the identity maps.

\[id\colon \mbox{Fun}(\cI,\cM)_{proj} \leftrightarrows \mbox{Fun}(\cI,\cM)_{inj} \colon id,\]
which is a Quillen equivalent because both model structures have the same weak equivalences. 
\item  Given a Quillen adjunction between combinatorial model structures $F\colon \cM \leftrightarrows \cN \colon G$ and a small category $\mathcal{I}$ the adjunction induced in the categories of functors $F^{\cI}\colon \rm{Fun}(\cI,\cM) \leftrightarrows\rm{Fun}(\cI,\cN) \colon G^{\cI}$ is a Quillen adjunction with respect to either the projective or the injective model structures. Furthermore if the $(F,G)$ is a Quillen equivalence, then so is  $(F^{\cI},G^{\cI})$. 

\item Let $f\colon \cI\rightarrow \cJ$ be a functor between small categories. Then the composition with $f$ induces a pullback functor  $f^{*}\colon \mbox{Fun}(\cI,\cM) \rightarrow \mbox{Fun}(\cI,\cM) $. Since $\cM$ admits small limits and colimits, $f^{*}$ has a right adjoint, which is denoted as $f_{*}$, and a left adjoint which is denoted as $f_{!}$. 
\end{itemize}
\end{rmk}

\begin{proposition}\label{basechange}
Let $\cM$ be a combinatorial model category and  $f\colon \cI\rightarrow \cJ$ functor between small categories. Then 
\begin{itemize}
\item The pair $(f_{!},f^{*})$ is a Quillen adjunction  between the projective model structures on $\mbox{Fun}(\cI,\cM)$ and $Fun(\cJ,\cM)$.
\item The pair $(f^{*},f_{*})$ is a Quillen adjunction  between the injective model structures on $\mbox{Fun}(\cI,\cM)$ and $Fun(\cJ,\cM)$.  

\end{itemize}
\end{proposition}
\begin{proof}
This follows from the fact that $f^{*}$ preserves level-wise weak equivalences, projective fibrations and injective cofibrations. 
\end{proof}

\subsection{Homotopy limits and homotopy colimits}

Let $[0]$ be the category with one object and one identity morphism. Let $f\colon \cI\rightarrow [0]$ be the unique functor  and $\cM$  a category which admits small limits and colimits. the pullback functor $\Delta:=f^{*}\colon \cM\rightarrow \mbox{Fun}(\cI,\cM)$, which sends every object $m\in\cM$ to the constant functor admits left and right adjoint. 
\begin{definition}
The right adjoint to $\Delta$ is called the \textit{limit functor} $\lim_{\cI}\colon \mbox{Fun}(\cI,\cM)_{inj}\rightarrow \cM$ and the  left adjoint is called the \textit{colimit functor} $\mbox{colim}_{\cI}\colon \mbox{Fun}(\cI,\cM)_{proj}\rightarrow \cM$ 
\end{definition}

\begin{rmk}
\rm By the Proposition \ref{basechange} the pair of adjoint functors \, $\mbox{colim}_{I}\colon \mbox{Fun}(\cI,\cM)_{proj}\leftrightarrows \cM\colon \Delta$    \, and   \, $ \Delta \colon\cM\leftrightarrows \mbox{Fun}(\cI,\cM)_{inj} \colon \lim_{I} $ \, are Quillen pairs.
\end{rmk}

\begin{definition}
Let $\cI$ be an small category and $\cM$ a combinatorial model category. The \textit{homotopy limit functor} is the right derived functor $\mathbf{R}\mbox{lim}_{\cI}$ and the \textit{homotopy colimit functor} is the left derived functor $\mathbf{L}\mbox{colim}_{I}$
\end{definition}

\begin{definition}
A monoidal model structure on a closed symmetric monoidal category $(\cM,\otimes,\mathbb{I})$ is model stricture on $\cM$ such that the following properties are fulfilled.
\begin{itemize}
\item \it Push-out product axiom. \rm For every pair of cofibrations $i:A\hookrightarrow B$ and $j:C\hookrightarrow D$ their push-out product 
 \[i\scalebox{.74}{$\square$} j:B\otimes C\coprod_{A\otimes C}A\otimes D\rightarrow B\otimes D \] is also a cofibration. If in addition one of the former morphisms is a \it weak-equivalence \rm, so is the later morphism. 
 
  \item \it Unit Axiom. \rm For every cofibrant object $A$ and for a cofibrant replacement $Q(\{*\})$, the induced map $Q(\{*\})\otimes A\rightarrow \{*\}\otimes A$ is a weak equivalence. 
 \end{itemize}
 \end{definition}
 \begin{rmk}
 \item This is equivalent to saying that for every symmetric monoidal model category and a cofibrant object $A$ the  adjunction $(-\otimes A, \underline{\mbox{hom}}(A,-))$ is a Quillen adjunction.  \end{rmk}

\bibliographystyle{amsplain}
\bibliography{Bibliografia}

\end{document}